\documentclass[11pt]{amsart}
\usepackage{amsmath,amssymb,array,latexsym,amsthm,pdfsync,hyperref,enumerate}
\usepackage{color}
\newtheorem{secthm}{Theorem}[section]
\newtheorem{seclem}{Lemma}[section]
\newtheorem{secrem}{Remark}[section]
\newtheorem{seccor}{Corollary}[section]
\newtheorem{secdefn}{Definition}[section]
\newtheorem{subsecthm}{Theorem}[subsection]
\newtheorem{subseclem}{Lemma}[subsection]
\newtheorem{subsecrem}{Remark}[subsection]

\numberwithin{equation}{section}
\newtheorem{secthmA}{Theorem A\!\!}

\newtheorem{prop}{Proposition}

\def\E{\mathbb{E}}
\def\a{\alpha}
\def\b{\beta}
\def\notin{\epsilon \hspace{-0.37em}/}

\def\rd{\color{red}}

\date{2-17-2018}                                           
\begin{document}

\title{Limits for Partial Maxima of Gaussian Random Vectors}
\author{James Kuelbs}
\address{James Kuelbs\\ Department of Mathematics,  University of Wisconsin, Madison, WI 53706-1388} 
\email{kuelbs@math.wisc.edu}

\author{Joel Zinn}
\address{Joel Zinn\\ Department of Mathematics, Texas A\&M University, College Station, TX 77843-3368} 
\email{joel.zinn@gmail.com}

\begin{abstract}
We obtain almost sure limit theorems for partial maxima of norms of a sequence of Banach-valued Gaussian random variables. 
\end{abstract}

\maketitle

\section{Introduction}\label{intro}  Limit theorems for various maximal functions of a sequence of random variables or a continuous time process have a rich and extensive history. They include distributional results as well as almost sure results in a variety of settings, and here we obtain related results for Gaussian sequences and the Ornstein-Uhlenbeck process with values in a separable Banach space. The results we establish are almost sure limit theorems motivated by the work of Berman \cite{berman-LLN} and Pickands \cite{pickands} for real-valued random variables, which can be viewed as laws of large numbers when the maximal functions are appropriately centered. 
We also deal with some almost sure results related to the classical Darling-Erd\H os Theorem \cite{darling-erdos}. 

We now state a few of the results mentioned above, so the reader can have some points of comparison. Since the Darling-Erd\H os Theorem was the earliest of those in the general one-dimensional context, we'll start with that.\\

{\bf Notation}: Throughout the paper we take $Lx=: \max\{\log_{e}x, 1\}$.\\

Motivated by a paper of Robbins \cite{robbins-aspects}, Darling and Erd\H os \cite{darling-erdos} obtained the following distributional limit theorem. 
\begin{secthmA}\label{darling-erdos}Let $\{\xi, \xi_{j}\colon j\ge 1\}$ be iid rv's with $\E \xi =0, \E \xi^{2}=1$ and $\E |\xi |^{3}<\infty$. Further, let $S_{k}=\sum_{j=1}^{k}\xi_{j}$ and 
\[\a_{n}=(2LLn)^{1/2}\ \text{and}\ \b_{n}=(2LLn+\frac12LLLn-\frac12L(4\pi) ).\] 
Then for every $x\in\mathbb{R}$,
\begin{align}\label{DE eqn}&\lim_{n\to\infty}\Pr\big(\a_{n}(\max_{k\le n}k^{-1/2}S_{k}- \frac{\b_{n}} {\a_{n}})\le x\big)=\exp\{ -e^{-x}\}.
\end{align}
\end{secthmA}

\begin{secrem}Note that this differs from a corollary of the ``invariance principle'', which yields a distributional limit theorem for $\max_{k\le n}n^{-1/2}S_{k}$.
Hence, as one might expect, the result obtained in (\ref{DE eqn}) involves a number of intricate steps. The first  establishes the Gaussian case, and then the Berry-Esseen theorem is used to obtain the result just stated. Along the way a form of the invariance principle, the law of the iterated logarithm, and a comparison to the Ornstein-Uhlenbeck process are employed
\end{secrem}

The law of large numbers results by Berman \cite{berman-LLN} and followed by Pickands \cite{pickands} are as follows.

\begin{secthmA}(Berman 1962) Let $\{\eta_{j}\}$ be a stationary Gaussian process such that  $\E \eta_{j}=0, \E\eta^{2}_{j}=1$ and for $\E\eta_{1}\eta_{j}=r_{j},\ nr_{n}\to 0$. Then $\max_{k\le n}\eta_{k}-(2Ln)^{1/2}\to 0$, in probability.
\end{secthmA}

The results in Pickands extend this result to convergence a.s, and also obtains similar results for real-valued continuous time stationary Gaussian processes.

In the recent paper \cite{darling-erdos-euclidean} Dierickx and Einmahl establish a version of the Darling-Erd\H os Theorem for the Euclidean norm of sums of $\mathbb{R}^d$-valued random vectors. Thus one might ask whether similar results might hold for Banach-valued random variables. Some first results related to almost sure limits are obtained in Corollary 4.1 below, but even for Gaussian random vectors much remains unsettled. In fact, whether the Darling-Erd\H os theorem is valid in the setting we are studying remains open. The potential difficulties dealing with other norms on finite dimensional spaces seem considerable - much less norms on infinite dimensional spaces. Furthermore, since the Darling-Erd\H os theorem deals with (partial sums of) iid rv's under a moment condition with CLT normalizations, and CLT's in the infinite dimensional setting require restrictions on the geometry of the space, there are many additional difficulties. Hence, as was done in the study of limit theorems for Banach-valued rv's, it is natural to first consider the case of the law of large numbers in this setting. In addition, the nature of the explicit centerings and the importance of Gaussian random variables in the proofs, even for real-valued random variables and processes, suggests that tools developed for log and loglog laws for Banach-valued random vectors as in \cite{carmona-kono}, \cite{led-tal-LIL}, \cite{lepage-LIL} and \cite{goodman-kuelbs-weakly} may be of some use. We have found this to be the case for the law of large number results established in this paper,  but explicit distributional results such as those of the  Darling-Erd\H os theorem involve additional difficulties.

Goodman's paper \cite{goodman-normal} provided additional motivation for the questions addressed here, and the first part of Theorem 2.1 in \cite{goodman-normal} as it relates to the maxima of  the norm of i.i.d. Gaussian random vectors in the Banach space setting emerges as a special case of (\ref{lim Mn}) in Corollary \ref{limsup le 0 f0 stat}.

In Section \ref{strong law seq} we present some notation and strong law results for maxima of the norm of sequences of Banach-valued random vectors, and their proofs. Section \ref{strong law continuous} provides analogous results for continuous time processes, a corollary for stationary Gaussian processes in this setting, and the proofs of these results. Section \ref{banach OU} studies the Ornstein-Uhlenbeck process, and its implications for maximal functions of norms of normalized partial sums in the Banach space setting. These are quantities that would naturally appear in a Darling-Erd\H{o}s result for normalized partial sums of centered Gaussian random vectors. 

Section \ref{banach OU} also provides some applications to self-adjoint operator-valued Gaussian random vectors and their spectrums, and symmetrization results in Proposition \ref{centering}. When combined with the centerings in our law of large number results as in Proposition \ref{medians via symm}, these symmetrizations allow us to determine the asymptotic behavior of the various medians of these partial maxima. Determining such asymptotics by direct calculation appears far less promising.

Section \ref{OU} starts with details on the sample function continuity of the Banach-valued Ornstein-Uhlenbeck process. This allows us to examine this process in enough detail that the technical assumptions for the continuous time results of Section \ref{strong law continuous} emerge as consequences of the definition of the process itself.

\section{Strong Laws for Partial Maxima of Sequences}\label{strong law seq} 

Throughout this section $B$ is a separable Banach space over the reals with norm $\|\cdot \|$, and its dual space is denoted by $B^*$ with norm $\|\cdot \|^{*}$. A probability measure $\mu$ on the Borel subsets of $B$ is a centered Gaussian measure if every linear functional $h \in B^*$ has a Gaussian distribution with mean zero and variance
$\int_Bh^2(x)d\mu(x)$. Since $\mu$ is a centered Gaussian measure on $B$ we recall some necessary notation that can be found in Lemma 2.1 of \cite{kuelbs-strong-convergence}.  
That is, since $\mu$ is a centered Gaussian measure, then
\begin{align}\label{gaussian variance}&
\int_B \|x\|^{2}\, d\mu(x)< \infty. 
\end{align}
and there is a unique Hilbert space $H_{\mu}\subseteq B$ given by the completion of the linear space $S(B^{*})$ in the inner product
\begin{align}\label{S inner product}\langle Sf,Sg\rangle_{\mu} =\int_B f(x)(g)d\mu(x),  
\end{align}
where
\begin{align}\label{Sf}
Sf=: \int_B xf(x)d\mu(x), f \in B^*,
\end{align}
are Bochner integrals in $B$. Furthermore, Lemma 2.1 in \cite{kuelbs-strong-convergence} 
implies 
\begin{align}
\sigma(\mu)=: \sup_{\|f\|^{*} \le 1}(\int_{B} f^2(x)d\mu(x))^{1/2} < \infty,
\end{align}
\begin{align}
\|x\| \le \sigma(\mu)\|x\|_{\mu},\, x \in H_{\mu},
\end{align}
and the unit ball of $H_{\mu}$,
\begin{align}
K=:\{x \in H_{\mu}: \|x\|_{\mu}\le 1\}, 
\end{align}
is a compact subset of $B$ with $\Gamma=:\sup_{x \in K}\|x\|$ finite. The Hilbert space $H_{\mu}$ is often referred to as the Hilbert space that generates the Gaussian measure $\mu$.

The following lemma further explains the notation used in our results. It links $\Gamma$ to $\sigma(\mu)$ in (\ref{gamma=sigma}) and also to the variance of a single linear functional in (\ref{int f_{0}^{2}}). This is relevant in that it links the normalizations of the terms maximized in our limit theorems to those one would use for i.i.d. real-valued centered Gaussian variables. It appears in connection with the assumptions (\ref{liminf max fk ge 0}) and (\ref{liminf max f0 ge 0}) we use to obtain the liminf in our strong law results for maximuma of sequences of centered Gaussian random vectors, and at least in Corollary \ref{limsup le 0 f0 stat} the lemma allows us to simplify these cumbersome assumptions because of what is known for real-valued Gaussian random variables.

\begin{seclem}\label{link gamma sigma} Let $\mu$ be a mean zero Gaussian measure on $B$ with norm $\|\cdot \|$, and assume $H_{\mu}$ is as above with (2.2)-(2.6) 
holding. If $\Gamma= \sup_{ x \in K} \|x\|,$ then 
\begin{align}\label{gamma=sigma}
\Gamma =\sigma(\mu).
\end{align}
Moreover, if $f_0 \in B^*$ and $ x_0 \in K$ with $\|x_0\|=\Gamma, f_0(x_0)= \Gamma,$ and $\|f_0\|^*=1$, then 
\begin{align}\label{int f_{0}^{2}}
(\int_B f_{0}^{2} (x) d\mu(x))^{1/2}=\Gamma.
\end{align}
Furthermore, if we define $\tilde {q}(x) = \|x\|/\Gamma, x \in B$, then $\tilde {q}(\cdot)$ is a norm on $B$ with dual norm  $\tilde{q}^*(f) = \sup_{\{ x: \tilde{q}(x)\le 1\}}f(x), f \in B^*$, and
\begin{align}\label{Gamma q tilde normalization}\Gamma_{\tilde{q}}=& \sigma(\mu)_{\tilde{q}}=1,
\end{align}
where
\begin{align}&\Gamma_{\tilde q}=: \sup_{x \in K} \tilde{q}(x)~{\rm{and}}~\sigma(\mu)_{ \tilde {q}} =: \sup_{{\tilde {q}}^*(f) \le 1}(\int_{B} f^2(x)d\mu(x))^{1/2} < \infty. 
\end{align}
In addition, if $f_0 \in B^*$ and $ x_0 \in K$ with $\tilde{q}(x_0)=1, f_0(x_0)=1$, and $ \tilde{q}^{*}( f_0)=1$, then
\begin{align}\label{f0 normalization}
(\int_B f_{0}^{2} (x) d\mu(x))^{1/2}=1.
\end{align}
\end{seclem}

\begin{secrem} 
Since $K$ is compact in $B$, the Hahn-Banach theorem gives us $f_0 \in B^*$ and $ x_0 \in K$ with $\|x_0\|=\Gamma, f_0(x_0)= \Gamma,$ and $\|f_0\|^*=1$. 
For all such choices of $f_0$ and $x_0$ the point of the lemma is that one then always has
\begin{align*}
(\int_B f_{0}^{2} (x) d\mu(x))^{1/2}=\Gamma. 
\end{align*} 
Moreover, in view of (\ref{Gamma q tilde normalization}) and (\ref{f0 normalization}) there is no loss of generality in assuming that the norm $\|\cdot\|$ on $B$ is such that 
\begin{align}\label{Gamma=1}
\Gamma= \sup_{x \in K} \|x\|=1.
\end{align}
\end{secrem}

\begin{proof}[Proof of Lemma \ref{link gamma sigma}] For each $f \in B^*$, (2.4) in \cite{kuelbs-strong-convergence} implies
\begin{align*}
\sup_{x \in K} f(x) =(\int_B f^2(y)d\mu(y))^{\frac{1}{2}}.
\end{align*}
Taking the sup over all $ f\in B^*$ with $||f||^* =1$, and interchanging sups on the left term we immediately have (\ref{gamma=sigma}). Assuming the conditions on $f_0 \in B^*$ and $x_0 \in K$ for (\ref{int f_{0}^{2}}) we also have
\begin{align*}
(\int_{B} f_0^2(y)d\mu(y))^{\frac{1}{2}} = \sup_{x \in K} f_0(x) \le \sup_{x \in K}||x|| = \Gamma~{\rm{and}}~\sup_{x \in K} f_0(x) \ge f_0(x_0) =\Gamma.
\end{align*}
Thus (\ref{int f_{0}^{2}}) holds, and when combined with (\ref{gamma=sigma}) the remainder of the lemma in (\ref{Gamma q tilde normalization}) and (\ref{f0 normalization}) also holds.
\end{proof}

\begin{secthm}\label{main with lim} Let $\mu, \mu_1,\mu_2,\cdots$ be  centered non-degenerate Gaussian measures on $B$ with norm $\|\cdot\|$, and $X,X_1,X_2,\cdots$ be $B$-valued random vectors on some probability space with distributions $\mu, \mu_1,\mu_2,\cdots$. In addition, assume
\begin{align}
\Gamma= \sup_{x\in K} \|x\|~{\rm{and}}~\Gamma_n= \sup_{x\in K_n} \|x\|, n \ge1,
\end{align}
where $K,K_1,K_2,\cdots$ are the unit balls of the Hilbert spaces $H_{\mu}, H_{\mu_1}, \cdots$ that generate the Gaussian measures $\mu, \mu_1,\mu_2,\cdots$, and for $n \ge1$ that
\begin{align}
\tilde{M}_n=: \max_{1 \le k \le n}\frac{\|X_k\|}{\Gamma_k}~{\rm{and}~} M_n=: \max_{1 \le k \le n}\frac{\|X_k\|}{\Gamma}. 
\end{align}
If $\{\mu_n: n \ge 1\}$ is assumed to converge weakly to $\mu$ in $B$, then
\begin{align}\label{gamma_{n}to gamma}
\lim_{n \rightarrow \infty} \Gamma_n=\Gamma>0,
\end{align}
and 
with probability one 
\begin{align}\label{limsup ge 0}
\limsup_{n \rightarrow \infty}[\tilde{M}_n - \sqrt{2Ln}] \le 0. 
\end{align}
Moreover, if 
\begin{align}\label{gamma ratio}
\frac{\Gamma_k}{\Gamma}-1=o((\sqrt{Lk})^{-1}), 
\end{align}
then with probability one
\begin{align}\label{gamma ratio le 0}
\limsup_{n \rightarrow \infty} [M_n - \sqrt{2Ln}] \le 0. 
\end{align}
In addition, if $f_n \in B^*,~ x_n \in K_n$ with $\|x_n\|= \Gamma_n, \|x_n\|_{\mu_n}=1$, $f_n(x_n)=\Gamma_n $, and $\|f_n\|^{*}=1$, then
\begin{align*}
\sigma_{f_n}^2 =: \int_B f_n^2(x)d\mu_n(x)=\Gamma_n^{2},
\end{align*}
and if
\begin{align}\label{liminf max fk ge 0}
\liminf_{n \rightarrow \infty} [\max\{f_1(\frac{X_1}{\Gamma_1}),\cdots,f_n(\frac{X_n}{\Gamma_n})\} - \sqrt{2Ln}] \ge 0 
\end{align}
with probability one, we also have
\begin{align}\label{liminf Mtilde ge 0}
\liminf_{n \rightarrow \infty} [\tilde{M}_n- \sqrt{2Ln}] \ge 0 
\end{align}
with probability one and
\begin{align}\label{liminf ge 0}
\liminf_{n \rightarrow \infty} [M_n - \sqrt{2Ln}] \ge 0.
\end{align}
with probability one whenever (\ref{gamma ratio}) is assumed.
\end{secthm}

An immediate corollary of Theorem \ref{main with lim} is the following.

\begin{seccor}\label{equal dist le 0} Let $\mu$ be a centered Gaussian measure on $B$ with norm $\|\cdot\|$, and assume $X_1,X_2,\cdots$ are defined on some probability space with each having distribution $\mu$. If $\Gamma=1$ as in (\ref{Gamma=1}), and
\begin{align*}
M_n= \max\{\|X_1\|,\cdots,\|X_n\|\}, n \ge 1,
\end{align*}
then with probability one
\begin{align}\label{limsup le 0 ident dist}
\limsup_{n \rightarrow \infty} [M_n - \sqrt{2Ln}] \le 0.
\end{align}
In addition, if $f_0 \in B^*,~ x_0 \in K$ with $\|x_0\|= \|x_0\|_{\mu}=1$, $f_0(x_0)=1$, and $\|f_0\|^{*}=1$, then
\begin{align*}
\sigma_{f_0}^2 =: \int_B f_0^2(x)d\mu(x)=1,
\end{align*}
and if
\begin{align}\label{liminf max f0 ge 0}
\liminf_{n \rightarrow \infty} [\max\{f_0(X_1),\cdots,f_0(X_n)\} - \sqrt{2Ln}] \ge 0 
\end{align}
with probability one, we also have
\begin{align}
\liminf_{n \rightarrow \infty} [M_n- \sqrt{2Ln}] \ge 0 
\end{align}
with probability one.
\end{seccor}

\begin{secrem}\label{individual gaussian} In Theorem \ref{main with lim} and Corollary \ref{equal dist le 0} it is not assumed the random vectors $\{X_n: n \ge 1\}$ are jointly Gaussian, only that each is Gaussian.
The limsup results of (\ref{limsup ge 0}) and (\ref{limsup le 0 ident dist}) are obtained using a Borel-Cantelli argument that is  based on rates of convergence results for clustering and convergence of $X_n$ to the set $K_n$ obtained in \cite{goodman-kuelbs-weakly}. Combined with (\ref{gamma ratio}) this sort of argument also yields (\ref{gamma ratio le 0}). For real-valued $\{X_n: n \ge 1\}$ this is fairly simple, but in the Banach space setting it is far less so. In contrast, the liminf results in the real-valued case are considerably more complex, but here they follow easily using Lemma \ref{link gamma sigma} and the assumptions (\ref{liminf max fk ge 0}) and (\ref{liminf max f0 ge 0}), which are likely to be hard (maybe even impossible) to verify in many settings. Situations where they can be simplified through a combination of Lemma \ref{link gamma sigma} and \cite{pickands} appear in Corollary \ref{limsup le 0 f0 stat}. Of course, using Lemma \ref{link gamma sigma} the liminf results for for i.i.d. sequences can also be done directly.
\end{secrem} 

The stationary case in Corollary \ref{limsup le 0 f0 stat} also appears in connection with results for the vector-valued Ornstein-Uhlenbeck process presented in the following sections. First we need a couple of definitions.

\begin{secdefn} A sequence of $B$-valued random vectors $ \{X_n: n \ge 1\}$ is stationary if for all integers $r\geq 1, h \ge 1$ the finite dimensional distributions of
\begin{align}
(X_1,\cdots,X_r)~{\rm{and}}~(X_{1+h},\cdots,X_{r+h})~{\rm{on}}~B^{r}
\end{align}
are equal.
\end{secdefn}

\begin{secdefn} A sequence of $B$-valued random vectors $ \{X_n: n \ge 1\}$ is a mean zero Gaussian sequence if for all integers $d\geq 1$ the finite dimensional distribution of
\begin{align}
(X_1,\cdots, X_d) 
\end{align}
is a mean zero Gaussian measure on $B^d$. 
\end{secdefn}

\begin{seccor}\label{limsup le 0 f0 stat} Let $\mu$ be a centered Gaussian measure on $B$ with norm $\|\cdot\|$, and assume $\{X_n:n\ge 1\}$ is a sequence of random vectors on some probability space with $\mathcal{L}(X_n)=\mu$ for $n \ge 1$. If $\Gamma=1$ as in (\ref{Gamma=1}), and
\begin{align*}
M_n= \max\{\|X_1\|,\cdots,\|X_n\|\}, n \ge 1,
\end{align*}
then with probability one
\begin{align}\label{limsup le 0 conclusion}
\limsup_{n \rightarrow \infty} [M_n - \sqrt{2Ln}] \le 0. 
\end{align}
Furthermore, if for some $f_0 \in B^*$ such that  $\|f_0\|^{*}=1$  the sequence $\{f_0(X_n): n \ge 1\}$ is a stationary mean zero Gaussian sequence with
\begin{align*}
\sigma_{f_0}^2 =: \int_B f_0^2(x)d\mu(x)=1,
\end{align*}
and 
\begin{align}\label{normalize f0}
\lim_{n \rightarrow \infty} ( \log_{e}n)E[f_0(X_1)f_0(X_n)] = 0, 
\end{align}
then
\begin{align}\label{liminf ge 0 f0 stat}
\liminf_{n \rightarrow \infty} [M_n- \sqrt{2Ln}] \ge 0 
\end{align}
with probability one. In particular, if $\{X_n:n\ge 1\}$ are i.i.d. with $\Gamma=1$, then with probability one
\begin{align}\label{lim Mn}
\lim_{n \rightarrow \infty} [M_n- \sqrt{2Ln}] = 0.
\end{align} 
\end{seccor}

\begin{secrem} If we assume $\{X_n: n \ge 1\}$ is a centered stationary Gaussian sequence in Corollary \ref{limsup le 0 f0 stat}, then $\{f_0(X_n): n \ge 1\}$ is  a mean zero stationary real-valued Gaussian sequence for all $f_0 \in B^*$. Therefore, if we also have $\|f_0\|^*=1$ with $f_0(x_0)=1$ for some $x_0 \in K$, then Lemma \ref{link gamma sigma} implies $\sigma^2(f_0)=1$, and  (\ref{normalize f0}) then implies (\ref{liminf ge 0 f0 stat}).
\end{secrem}

\begin{proof}[Proof of Theorem \ref{main with lim} and its Corollaries] Since the centered Gaussian measures  $\{\mu_k: k\ge 1\}$ are non-degenerate and converge weakly to the non-degenerate measure $\mu$ on $B$, (2.2) in Theorem 1 of \cite{goodman-kuelbs-weakly} implies (\ref{gamma_{n}to gamma}). Furthermore, for $\epsilon>0$ and $\epsilon_k= \epsilon/\sqrt{2Lk}$, under the assumptions of Theorem \ref{main with lim}, (2.3) in Theorem 1 of \cite{goodman-kuelbs-weakly} implies that with probability one
\begin{align}
\frac{X_k(\omega)}{\sqrt{2Lk}} \in K_k + \epsilon_k\, U
\end{align}
for all $k \ge k_0(\omega,\epsilon)$. This implies
\begin{align}\label{le Gammak +}
\frac{\|X_k(\omega)\|}{\sqrt{2Lk}} \le \Gamma_k + \epsilon_k, ~{\rm{and ~hence ~that}}~ \frac{\|X_k(\omega)\|}{\Gamma_k} \le \sqrt{2Lk} + \epsilon 
\end{align}
for all $k \ge k_0=:k_0(\omega,\epsilon)$ with probability one. Since
\begin{align*}
\tilde{M}_n(\omega)= \max_{1 \le k \le n} \frac{\|X_k(\omega)\|}{\Gamma_k},
\end{align*}
(\ref{limsup ge 0}) is immediate for all $\omega$ such that $\sup_{n \ge 1}\tilde{M}_n(\omega)< \infty$. If \hfil\break
$\sup_{n \ge 1}\tilde{M}_n(\omega)= \infty$, then for $k_0=k_0(\omega,\epsilon)$ and all $n \ge n_0(\omega,\epsilon)$
\begin{align*}
\tilde{M}_n(\omega)= \max_{k_0 \le k \le n} \frac{\|X_k(\omega)\|}{\Gamma_k} \le \sqrt{2Ln} + \epsilon,
\end{align*}
and, since $\epsilon>0$ is arbitrary, for such $\omega$ we have (\ref{limsup ge 0}). Thus with probability one we have (\ref{limsup ge 0}), and the limsup results in Corollaries \ref{equal dist le 0} and \ref{limsup le 0 f0 stat}, namely (\ref{limsup le 0 ident dist}) and (\ref{limsup le 0 conclusion}), also hold.

Now we turn to the proof of the liminf results in (\ref{liminf Mtilde ge 0}), 
and the implications for the liminf results of Corollaries \ref{equal dist le 0} and \ref{limsup le 0 f0 stat}. Given the assumptions on $\{f_n: n \ge 1\}$ in Theorem \ref{main with lim}, Lemma \ref{link gamma sigma} implies $\sigma_{f_{n}}^{2}= \Gamma_{n}^{2}$, and with probability one
\begin{align*}
M_n(\omega) \ge \max_{1 \le k \le n} \frac{f_k(X_k(\omega))}{\Gamma_k}.
\end{align*}
That  (\ref{liminf Mtilde ge 0}) holds with probability one is now  immediate from (\ref{liminf max fk ge 0}). The reader will note we have not used the fact that we have $\sigma_n^2= \Gamma_n^2$, but we have included it since it is the normalization required to verify the analogue of assumption (\ref{liminf max fk ge 0}) in 
Corollary \ref{limsup le 0 f0 stat}. An entirely similar argument also gives the liminf result in (\ref{liminf max f0 ge 0}), so  Corollary \ref{equal dist le 0} is proven.

The liminf result in (\ref{liminf ge 0 f0 stat}) of Corollary \ref{limsup le 0 f0 stat} follows from Theorem 3.3 in \cite{pickands} since $\sigma_{f_0}^2 =1$, and (\ref{normalize f0}) is assumed hold. To verify (\ref{lim Mn}) observe that $\Gamma=1$, and Lemma \ref{link gamma sigma} implies there exists $f_0 \in B^*$ with $\|f_0\|^*=1, \sigma_{f_0}^2=1$, and such that $\{f_0(X_k): k \ge 1\}$ is a sequence of i.i.d Gaussian random variables with mean zero and variance one. Hence (\ref{normalize f0}) of 
Corollary \ref{limsup le 0 f0 stat} is trivial, and (\ref{liminf ge 0 f0 stat}) implies the liminf result for (\ref{lim Mn}). Since the limsup result follows from (\ref{limsup le 0 conclusion}), this proves Corollary \ref{limsup le 0 f0 stat}.

What remains in the proof of Theorem \ref{main with lim} is to verify the limsup in (\ref{gamma ratio le 0}) and the liminf in (\ref{liminf ge 0}) hold with probability one when (\ref{gamma ratio}) is assumed. For $\epsilon>0$ and $\epsilon_k= \epsilon/\sqrt{2Lk}$, from (\ref{le Gammak +}) we have for all  $\omega$ in a set of probability one and $k \ge k_0(\omega,\epsilon)$ that 
 \begin{align*}
\frac{\|X_k(\cdot,\omega)\|}{\sqrt{2Lk}} \le \Gamma +(\Gamma_k - \Gamma) + \epsilon_k, 
\end{align*}
which implies
\begin{align}\label{over Gamma le + epsilon over Gamma}
\|X_k(\cdot,\omega)\|/\Gamma \le \sqrt{2Lk} +(\frac{\Gamma_k}{\Gamma} -1) \sqrt{2Lk} +\frac{\epsilon}{\Gamma} \end{align}
for all $k \ge k_0(\omega,\epsilon)$ with probability one Since we are assuming (\ref{gamma ratio}), there exists non-random $k_1$ such that $k \ge k_1(\epsilon)$ implies
\begin{align*}
|\frac{\Gamma_k}{\Gamma} -1)| \sqrt{2Lk} \le \epsilon,
\end{align*}
and combining with (\ref{over Gamma le + epsilon over Gamma}) this implies 
\begin{align}
\|X_k(\cdot,\omega)\|/\Gamma \le \sqrt{2Lk} +\epsilon +\frac{\epsilon}{\Gamma} 
\end{align}
for all $k \ge k_2(\omega,\epsilon)=:\max\{k_0(\omega,\epsilon),k_1(\epsilon)\}$. Since  
\begin{align*}
\max_{1 \le k \le n}\frac{\|X_k(\cdot,\omega)\|}{\Gamma} 
\end{align*}
is increasing in $n$ and (\ref{gamma ratio le 0}) is trivial if it is bounded in $n$, we assume
\begin{align}\label{partial max = infinity}
\sup_{n\ge 1} \max_{1 \le k \le n}\frac{\|X_k(\cdot,\omega)\|}{\Gamma}= \infty, 
\end{align}
which implies
\begin{align}
 \limsup_{n \rightarrow \infty}[ \max_{1 \le k \le n} \frac{\|X_k(\cdot,\omega)\|}{\Gamma}- \sqrt{2Ln}] \\
=\limsup_{n \rightarrow \infty}[ \max_{k_2(\omega{\rd ,}\epsilon) \le k \le n} \frac{\|X_k(\cdot,\omega)\|}{\Gamma}- \sqrt{2Ln}] \le \epsilon(1 +1/\Gamma)\notag
\end{align} 
whenever (\ref{partial max = infinity}) holds. Since $\epsilon>0$ is arbitrary, when combined with the trivial case, this implies (\ref{gamma ratio le 0}) with probability one.

Now we turn to the final step in the proof, which is to verify (\ref{liminf ge 0}). Given the assumptions on the linear functionals $\{f_k\}$, we have 
\begin{align*}
 [ \max_{1 \le k \le n} \frac{\|X_k(\cdot,\omega)\|}{\Gamma}- \sqrt{2Ln}] \ge[ \max_{1 \le k \le n} \frac{|f_k(X_k(\cdot,\omega))|}{\Gamma}- \sqrt{2Ln}] ,
\end{align*}
so it suffices to show that the assumption (\ref{liminf max fk ge 0}) implies that 
\begin{align}
\liminf_{n \rightarrow \infty} [ \max_{1 \le k \le n} \frac{|f_k(X_k(\cdot, \omega))|}{\Gamma}- \sqrt{2Ln}] \ge 0 
\end{align}
with probability one. Since (\ref{liminf max fk ge 0}) implies
\begin{align*}
\sup_{n \ge 1} \max_{1 \le k \le n} \frac{|f_k(X_k(\cdot, \omega))|}{\Gamma_k}
\end{align*}
increases to infinity with probability one, for every $k_0 \ge 1$ we have for $n \ge n_0(\omega,k_0)$ that
\begin{align}\label{max=max}
 \max_{1 \le k \le n} \frac{|f_k(X_k(\cdot, \omega))|}{\Gamma_k} =  \max_{k_0 \le k \le n} \frac{|f_k(X_k(\cdot, \omega))|}{\Gamma_k} 
\end{align}
with probability one, and (\ref{liminf max fk ge 0}) implies
\begin{align}\label{liminf fk}
\liminf_{n \rightarrow \infty} [ \max_{k_o \le k \le n} \frac{|f_k(X_k(\cdot, \omega))|}{\Gamma_k}- \sqrt{2Ln}] \ge 0 
\end{align}
with probability one. Now for each $\epsilon>0$, (\ref{gamma ratio}) allows us to choose $k_1=:k_1(\epsilon) \ge 1$ such that $k\ge k_1$ implies
\begin{align*}
|\frac{\Gamma_k}{\Gamma} -1|\sqrt{2Lk}<\epsilon.
\end{align*}
In addition, note that
\begin{align*}
\frac{|f_k(X_k(\cdot, \omega))|}{\sqrt{2Lk}\, \Gamma_k} \le 2
\end{align*}
for all $k \ge k_2(\omega)$ with probability one,  and
\begin{align*}
 \frac{|f_k(X_k(\cdot, \omega))|}{\Gamma}= \frac{|f_k(X_k(\cdot, \omega))|}{\Gamma_k}\frac{\Gamma_k}{\Gamma}\\= \frac{|f_k(X_k(\cdot, \omega))|}{\Gamma_k}+  \frac{|f_k(X_k(\cdot, \omega))|}{\Gamma_k}(\frac{\Gamma_k}{\Gamma}-1).
 \end{align*}
Therefore, for $k \ge k_0=k_0(\omega,\epsilon) \ge \max\{k_1(\epsilon), k_2(\omega)\}$
\begin{align*}
 \frac{|f_k(X_k(\cdot, \omega))|}{\Gamma}\ge \frac{|f_k(X_k(\cdot, \omega))|}{\Gamma_k}-2\epsilon, 
\end{align*}
so for $n \ge n_0(\omega,k_0(\omega,\epsilon))$ we have
\begin{align}\label{over Gamma over Gammak}
 \max_{1 \le k \le n} \frac{|f_k(X_k(\cdot, \omega))|}{\Gamma} &\ge \max_{k_0 \le k \le n} \frac{|f_k(X_k(\cdot, \omega))|}{\Gamma} \\
 &\ge \max_{k_0 \le k \le n} \frac{|f_k(X_k(\cdot, \omega))|}{\Gamma_k} -2\epsilon\notag \\
&=\max_{1 \le k \le n} \frac{|f_k(X_k(\cdot, \omega))|}{\Gamma_k} -2\epsilon,  \notag
\end{align}
where the equality follows from (\ref{max=max}) and our choice of $n$. Therefore, (\ref{liminf fk}) and (\ref{over Gamma over Gammak}) imply
\begin{align}
\liminf_{n \rightarrow \infty} [ \max_{1 \le k \le n} \frac{|f_k(X_k(\cdot, \omega))|}{\Gamma}- \sqrt{2Ln}] \ge \\
  \liminf_{n \rightarrow \infty} [ \max_{1 \le k \le n} \frac{|f_k(X_k(\cdot, \omega))|}{\Gamma_k}- \sqrt{2Ln}] - 2\epsilon \ge -2\epsilon.\notag
\end{align}
with probability one. Since $\epsilon>0$ was arbitrary this implies (\ref{liminf ge 0}), and Theorem \ref{main with lim} is proved.
\end{proof}

\section{Strong Laws for Maxima of Continuous Time Processes}\label{strong law continuous} 

Applying Theorem \ref{main with lim} and its corollaries we obtain generalizations to continuous time vector-valued stochastic processes. In particular, Corollary \ref{limsup le 0 f0 stat} allows us to provide some generalizations of results for real-valued stationary Gaussian processes that appeared in \cite{pickands}, and the references therein. In the real-valued case the proofs of the continuous time results are more complex, so it is somewhat of a surprise that the sequence results in section two make at least part of the argument easier even for Banach-valued sample continuous Gaussian processes (see Remark \ref{individual gaussian} for further clarification and details). Now we need some additional notation.

Throughout this section $E$ has norm $q(\cdot)$ and $B=C_E[0,1]$ denotes the space of $E$-valued continuous functions on $[0,1]$ with norm 
\begin{align}\label{norm C_{E}}
\|x\|= \sup_{t \in [0,1]} q(x(t)), x\in C_E[0,1].
\end{align}

 Let $Y=: \{Y(t): t \ge 0\}$ denote a centered, sample continuous process with values in $(E,q(\cdot))$, and 
for each integer $k\ge 1$ define the processes
\begin{align}\label{X_{k}}
X_k(t)=:Y(t+(k-1)), t \in [0,1]. 
\end{align}
Then, $\{X_k(t): t \in [0,1]\}$ is a sample continuous process, and its distribution is a centered measure $\mu_k$ on $B=C_E[0,1]$ with norm as in (\ref{norm C_{E}}).

\begin{secdefn}\label{def gaussian} A stochastic process $Z=: \{Z(t): t \in T\}$ is said to be an $E$-valued mean zero Gaussian process if for each integer $d\geq 1$ and finite subset 
$\{ t_1,t_2,\cdots ,t_d\}$ of $T$ the finite dimensional distribution of
\begin{align}
(Z(t_1),\cdots,Z(t_d)) 
\end{align}
is a mean zero Gaussian measure on $E^d$. 
\end{secdefn}

If $Y=\{Y(t): t \ge 0\}$ is assumed to be a mean zero Gaussian process as in Definition \ref{def gaussian}, then the finite dimensional distributions of each of the processes $X_k$ are mean zero Gaussian. Since $B$ is separable in the sup-norm given in (\ref{norm C_{E}}), the Borel probability measures $\mu_k = \mathcal{L}(X_k)$ all have the same mean zero Gaussian finite dimensional distributions, but are the $\mu_k$ mean zero Gaussian measures on $B$? Recall that a measure $\mu $ is a mean zero Gaussian measure on a separable Banach space $B$ if every linear functional $f \in B^*$ is a mean zero Gaussian random variable with variance 
\begin{align}\label{var f}
\int_B f^2(x) d\mu(x). 
\end{align} 
The next lemma shows this is indeed the case.

\begin{seclem}\label{f(X) gaussian} If $X=:\{X(t): t \in [0,1]\}$ is an $E$-valued, mean zero, sample continuous Gaussian process per Definition \ref{def gaussian}, then $\mu = \mathcal{L}(X)$ is a Gaussian measure on the Borel subsets of $C_E[0,1]$. That is, for every $f \in C_E^{*}[0,1]$, $f(X)$ is a mean zero Gaussian random variable with variance as in (\ref{var f}) with $B=C_E[0,1]$.
\end{seclem}

\begin{proof} Let $X_1,X_2,\cdots, X_n$ be independent copies of $X$. Then, for each integer $d \ge 1$ and $0\le t_1<\cdots<t_d \le 1$ 
the finite dimensional distributions
\begin{align*}
\mu_X^{t_1,\cdots,t_d} =: \mathcal{L}(X(t_1),\cdots,X(t_d))
\end{align*}
and 
\begin{align*}
\mu_{X_{j}}^{t_1,\cdots,t_d} =: \mathcal{L}(X_j(t_1),\cdots,X_j(t_d)),
\end{align*}
on $E^d$ are such that
\begin{align*}
\mu_X^{t_1,\cdots,t_d}=\mu_{X_{j}}^{t_1,\cdots,t_d}. 
\end{align*}
Moreover, since the measures are mean zero Gaussian on $E^d$, we then have for 
\begin{align*}
Z_n =(X_1+\cdots+X_n)/\sqrt{n}.
\end{align*}
that 
\begin{align*}
\mu_{Z_n}^{t_1,\cdots,t_d} =: \mathcal{L}(Z_n(t_1),\cdots,Z_n(t_d))= \mu_X^{t_1,\cdots,t_d}~{\rm{on}}~E^d.
\end{align*}
Since equality of the finite dimensional distributions of measures on cylinder sets of $C_E[0,1]$ extends to the Borel subsets, we thus have
\begin{align*}
\mu=\mathcal{L}(X)= \mathcal{L}(Z_n), n \ge 1,
\end{align*}
and hence for every $f \in C_E^{*}[0,1]$ and $n \ge 1$
\begin{align*}
\mathcal{L}(f(X))= \mathcal{L}(f(Z_n)).
\end{align*}
Therefore, 
\begin{align*}
\mathcal{L}(f(X))= \mathcal{L}(\frac{f(X_1)+\cdots+f(X_n)}{\sqrt{n}}),
\end{align*}
and by Proposition 9.1 in \cite{breiman}, p. 186 and its extension in Problem 2 in \cite{breiman} p. 202,  $f(X)$ is mean zero Gaussian with variance as in (\ref{var f}). 
\end{proof}

\begin{secthm}\label{lim Gamma_{k}} Let $Y= \{Y(t): t \ge 0\}$ denote a centered $E$-valued sample continuous Gaussian process, and for each integer $k\ge 1$ define $X_k(t)$ as in (\ref{X_{k}}) with $\mu_k$ the centered Borel probability measure induced by $X_{k}$ on the Banach space $B=C_E[0,1]$ with norm
$\|\cdot\|$ given by (\ref{norm C_{E}}). Then, the measures $\{\mu_{k}:k \ge 1\}$ are Gaussian measures on $B$. In addition, assume $\mu$ is a non-degenerate Gaussian measure on $B$, and

\begin{align}
\Gamma= \sup_{x\in K} \|x\|~{\rm{and}}~\Gamma_k= \sup_{x\in K_k} \|x\|, k \ge1, 
\end{align}
where $K,K_1,K_2,\cdots$ are the unit balls of the Hilbert spaces $H_{\mu}, H_{\mu_1}, \cdots$ that generate the Gaussian measures $\mu, \mu_1,\mu_2,\cdots$, and 
\begin{align*}
\tilde{Y}(t)= \frac{Y(t)}{\Gamma_k}I(t\in [k-1,k)), k \ge 1,
\end{align*}
where $\tilde{Y}(t)$ is understood to be zero for $t \in [k-1,k)$ whenever $\Gamma_k=0$. If $\{\mu_k: k \ge 1\}$ converges weakly to $\mu$ in $B=C_E[0,1]$, then
\begin{align}\label{Gammak to Gamma}
\lim_{k \rightarrow \infty} \Gamma_k= \Gamma>0, 
\end{align}
and with probability one 
\begin{align}\label{limsup le 0 tilde Y}
\limsup_{T \rightarrow \infty} [\sup_{0 \le t \le T} q(\tilde{Y}(t)) - \sqrt{2LT}] \le 0.
\end{align}
Moreover, if 
\begin{align}\label{sqrt{(L k})^{-1}}
\frac{\Gamma_k}{\Gamma}-1=o((\sqrt{Lk})^{-1}),
\end{align}
then with probability one
\begin{align}\label{sqrt conclusion}
\limsup_{T \rightarrow \infty} [\sup_{0 \le t \le T} q(\frac{Y(t)}{\Gamma}) - \sqrt{2LT}] \le 0. 
\end{align}
In addition, if (\ref{liminf max fk ge 0}) holds with $f_n \in B^*,~ x_n \in K_n$ with $\|x_n\|= \Gamma_n, \|x_n\|_{\mu_n}=1$, $f_n(x_n)=\Gamma_n $, and $\|f_n\|^{*}=1$, then with probability one 
\begin{align}\label{liminf tilde Y ge 0}
\liminf_{T \rightarrow \infty} [\sup_{0 \le t \le T} q({\tilde{Y}}(t)) - \sqrt{2LT}] \ge 0, 
\end{align}
and
\begin{align}\label{liminf Y ge 0 Gamma k}
\liminf_{T \rightarrow \infty} [\sup_{0 \le t \le T} q(\frac{Y(t)}{\Gamma}) - \sqrt{2LT}] \ge 0. 
\end{align}
with probability one whenever (\ref{sqrt{(L k})^{-1}}) is assumed.
\end{secthm}

The proof of Theorem \ref{lim Gamma_{k}} follows easily from Theorem \ref{main with lim}, and is given below. However, we first indicate a corollary for continuous time stationary Gaussian processes. We start with a definition and a lemma.

\begin{secdefn}\label{def stationary} The process $Y=: \{Y(t): t \ge 0\}$ is said to be an $E$-valued stationary process if for each integer $r\geq 1,~ 0 \le t_1 <t_2<\cdots <t_r< \infty,$ and $h>0$ the finite dimensional distributions of
\begin{align}
(Y(t_1),\cdots,Y(t_r))~{\rm{and}}~(Y(t_1+h),\cdots,Y(t_r+h))~{\rm{on}}~E^{r} 
\end{align}
are equal.
\end{secdefn}

\begin{seclem}\label{Xk gaussian} If $Y=: \{Y(t): t \ge 0\}$ is a sample continuous mean zero $E$-valued stationary process, then the Borel probability measures $\mu_k= \mathcal{L}(X_k)$ on  $C_E[0,1]$ as in (\ref{X_{k}}) are equal. Moreover, if we also assume $Y=\{Y(t): t \ge 0\}$ is a mean Gaussian process in the sense of Definition \ref{def gaussian}, then the $X_k, k \ge 1,$ are mean zero Gaussian processes as in Definition \ref{def gaussian}. In addition, for every $f \in C_E^{*}[0,1]$, $f(X)$ is a mean zero Gaussian random variable with variance as in (\ref{var f}) with $B=C_E[0,1]$.
\end{seclem}

\begin{proof} For all $d\ge 1$ and $0 \le t_1<t_2<\cdots<t_d \le 1$ a typical cylinder set of $C_E[0,1]$ is 
\begin{align*}
A=\{ x \in C_E[0,1]: (x(t_1),\cdots,x(t_d))  \in J\}, 
\end{align*}
where $J$ is a Borel subset of $E^d$. The class of all such cylinder sets form an algebra of sets, and since $C_E[0,1]$ is separable the minimal sigma algebra containing them is all the Borel subsets. To see this recall the basic fact that for $\{t_j:j\ge 1\}$ a dense subset of [0,1] and $\epsilon>0$ we have
\begin{align*}
\{ x \in C_E[0,1]: \|x\|=&\sup_{t \in [0,1]}q(x(t) \le \epsilon\} \\
&= \cap_{n \ge 1} \{x \in C_E[0,1]: \sup_{ 1 \le j \le n}q(x(t_j)) \le \epsilon\},
\end{align*}
and then argue as is usual to show that open subsets of $C_E[0,1]$ are in this minimal sigma algebra. Since Y is assumed to be stationary, the finite dimensional distributions of $X_1$ agree with those for $X_k$ for all $k \ge 2$, which implies $\mu_k(A)=\mu_1(A)$ for all cylinder sets $A$, and hence $\mu_k= \mu_1$ on the Borel sets for all $k\ge 1$. Moreover, if we assume $Y$ is a mean zero $E$-valued Gaussian process, then the finite dimensional distributions  of $Y$ are all mean zero Gaussian, and hence those of the $X_k$ are also mean zero Gaussian for all $k \ge 1$. Therefore, the $X_k$ are mean zero $E$-valued Gaussian processes in the sense of Definition \ref{def gaussian}. Since
 $C_E[0,1]$ is separable in the sup-norm given in (\ref{norm C_{E}}), the remainder of the lemma follows from Lemma \ref{f(X) gaussian}.
\end{proof}

 \begin{seccor}\label{main with OU}  Let $Y= \{Y(t): t \ge 0\}$ be a centered $E$-valued sample continuous non-degenerate stationary Gaussian process, and 
for each integer $k\ge 1$ define $X_k(t)$ as in (\ref{X_{k}}) with $\mu_k$ the centered Borel probability measure induced by $X_k$ on $B=C_E[0,1]$ as given in (\ref{X_{k}}).
Then, the $\{\mu_k:k \ge 1\}$ are mean zero Gaussian measures on $B$ with $ \mu_k= \mu$ for all $k \ge 1$. Furthermore, if $K$ is the unit ball of the Hilbert space $H_{\mu}$ that generates the Gaussian measure $\mu$, and
\begin{align*}
\Gamma=\sup_{x \in K} \|x\|,
\end{align*}
then with probability one
\begin{align}\label{limsup Y le 0}
\limsup_{T \rightarrow \infty} [\sup_{0 \le t \le T} q(\frac{Y(t)}{\Gamma}) - \sqrt{2LT}] \le 0. 
\end{align}
Furthermore, if for some $f_0 \in B^*$ such that  $\|f_0\|^{*}=1$  the sequence $\{f_0(X_n): n \ge 1\}$ is a stationary mean zero Gaussian sequence with
\begin{align*}
\sigma_{f_0}^2 =: \int_B f_0^2(x)d\mu(x)=\Gamma^2,
\end{align*}
and 
\begin{align}\label{correl decay}
\lim_{n \rightarrow \infty} ( \log_{e}n)E[f_0(X_1)f_0(X_n)] = 0,
\end{align}
then
\begin{align}\label{liminf Y ge 0 corr}
 \liminf_{T \rightarrow \infty} [\sup_{0 \le t \le T} q(\frac{Y(t)}{\Gamma}) - \sqrt{2LT}] \ge 0.   
\end{align}
with probability one.
\end{seccor}

\begin{proof}[Proof of Theorem \ref{lim Gamma_{k}}] From Lemma \ref{f(X) gaussian} we have that the measures $\{\mu_k:k \ge 1\}$ are mean zero Gaussian measures on $B$, and since $\{\mu_k: k\ge 1\}$ converge weakly to the non-degenerate measure $\mu$ on $B=C_E[0,1]$,  (2.2) in Theorem 1 of  \cite{goodman-kuelbs-weakly} implies (\ref{Gammak to Gamma}). Furthermore, if $n-1\le T \le n$ then $LT -L(n-1) \le 1/(n-1)$ for $n \ge 2$,
so it suffices to prove the result when $T=n$. Then,
\begin{align*}
\sup_{0 \le t \le n}q(\tilde{Y}(t))= \sup_{k-1 \le t<k, 1 \le k \le n} q(\tilde{Y}(t))= \sup_{k-1 \le t<k, 1 \le k \le n} q(\frac{Y(t)}{\Gamma_k}), 
\end{align*}
which implies
\begin{align}\label{sup norms equal}
\sup_{0 \le t \le n}q(\tilde{Y}(t))= \sup_{0 \le t \le 1, 1 \le k \le n} q(\frac{X_k(t)}{\Gamma_k}) = \sup_{1 \le k \le  n} \frac{\|X_k(\cdot)\|}{\Gamma_k}, 
\end{align}
where sample function continuity of $X_k$ at $t=1$ is used on the last equality. Since we assumed (\ref{liminf max fk ge 0}) here, (\ref{sup norms equal}) , and (\ref{limsup ge 0}) and (\ref{liminf Mtilde ge 0}) of Theorem \ref{main with lim}, combine to establish (\ref{limsup le 0 tilde Y}) and (\ref{liminf tilde Y ge 0}) with $T=n$.
\bigskip

Since we also have
\begin{align}\label{sup=sup}
\sup_{0 \le t \le n}q(\frac{Y(t)}{\Gamma})= \sup_{0 \le t \le 1, 1 \le k \le n} q(\frac{X_k(t)}{\Gamma}) = \sup_{1 \le k \le  n} \frac{\|X_k(\cdot)\|}{\Gamma}, 
\end{align}
assuming (\ref{sqrt{(L k})^{-1}}) also holds, (\ref{sqrt conclusion}) and \ref{liminf Y ge 0 Gamma k})  follow from (\ref{gamma ratio le 0}) and (\ref{liminf ge 0}). Thus Theorem \ref{lim Gamma_{k}} is proved.
\end{proof}

\begin{proof}[Proof of Corollary \ref{main with OU}] As before it suffices to prove the results with $T=n$. The limsup result in (\ref{limsup Y le 0}) then follows using Lemma \ref{Xk gaussian}, and applying either (\ref{limsup le 0 tilde Y}) or (\ref{sqrt conclusion}) of Theorem \ref{lim Gamma_{k}}. To verify (\ref{liminf Y ge 0 corr}) observe from (\ref{sup=sup}) and that $f_0\in B^*$ with $\|f_0\|^*=1$ implies
 \begin{align*}
\liminf_{n \rightarrow \infty} [\sup_{0 \le t \le n} q(\frac{Y(t)}{\Gamma}) - \sqrt{2Ln}] = \liminf_{n \rightarrow \infty}[ \max_{1 \le k \le n} \frac{\|X_k(\cdot,\omega)\|}{\Gamma}- \sqrt{2Ln}] 
\end{align*}
\begin{align*}
\ge \liminf_{n \rightarrow \infty} [ \max_{1 \le k \le n} \frac{|f_0(X_k(\cdot, \omega))|}{\Gamma}- \sqrt{2Ln}] \ge 0 
\end{align*}
with probability one by applying Corollary \ref{limsup le 0 f0 stat} since $\sigma_{\frac{f_0}{\Gamma}}^{2}=1$ with $\{f_0(X_k)/\Gamma: k \ge 1\}$ a stationary Gaussian sequence of real-valued random variables with mean zero, variance one, and (\ref{var f})  is assumed. 
\end{proof}

\begin{secrem} In the next section we provide some some applications of the results in Sections 2 and 3. One example we consider in some detail is the Banach-valued Ornstein-Uhlenbeck process. In particular, we show all the assumptions in Corollary \ref{main with OU} (such as the stationarity of $\{f_0(X_k): k \ge 1\}$ and (\ref{correl decay})) can be verified directly from the process itself. Examples showing that if assumption (\ref{sqrt{(L k})^{-1}}) fails, then (\ref{sqrt conclusion}) need not hold are easy to find.
\end{secrem}

\section{Banach-valued Ornstein-Uhlenbeck Processes and Applications}\label{banach OU} 

The goal here is to show the Ornstein-Uhlenbeck process with values in a separable Banach space $E$ has a strong law of large numbers for its maximum as in Theorem \ref{Y stationary} below, and also to derive some strong law limit theorems for maximuma of normalized partial sums of Banach-valued Gaussian random vectors. In fact, we will show (\ref{correl decay}) (and hence (\ref{normalize f0})) can be verified to hold from the process itself, and is not an extra assumption for the Ornstein-Uhlenbeck process.

If $\gamma$ is a non-degenerate mean zero Gaussian measure on the Borel sets of $E$ with norm $q$, we let $W=:\{W(t): t \ge 0\}$ denote the Brownian motion in $E$ generated by $\gamma$.  
The stochastic process
\begin{align}\label{E-valued OU}
Y(t)=: e^{-\frac{t}{2}}W(e^{t}), t \ge 0,
\end{align}
is the $E$-valued Ornstein-Uhlenbeck process generated or determined by $\gamma$-Brownian motion. In particular, we assume $W$ is normalized so that the law of $W(1)$ is $\gamma$ ( see Subsection \ref{sample function 1} for more details).To simplify, we will sometimes say $Y=\{Y(t):t \ge 0\}$ is a 
$\gamma$-generated Ornstein-Uhlenbeck process. Since we always assume $\gamma$ is a non-degenerate mean zero Gaussian measure on $E$, its support is a closed linear subspace of $E$ of dimension at least one. Hence the $\gamma$-generated Ornstein-Uhlenbeck process is also always non-trivial.

The existence of a sample continuous $E$-valued Brownian motion $W=\{W(t): t \ge 0\}$ generated by $\gamma$ follows from \cite{gross-potential}. A precise description appears in Lemma \ref{existence of brownian} below, and more self contained proofs appear in the appendix for this paper \cite{partial-max-appendix}. This immediately implies the sample continuous Ornstein-Uhlenbeck process exists, and we assume throughout  $Y=:\{Y(t): t \ge 0\}$ as in (\ref{E-valued OU}) is a sample continuous version. Lemma (\ref{existence of brownian}) provides the construction of the process $W$ on the probability space $(\Omega_E, \mathcal{F},P)$, where $\Omega_E$ consists of the  
$E$-valued continuous functions $x$ defined on $[0,\infty)$ with $x(0)=0$, $\mathcal{F}$ is the $\sigma$-field of $\Omega_E$ generated by the functions $x \rightarrow x(t), 0 \le t < \infty$, and $P$ is the probability measure on $(\Omega_E,\mathcal{F})$ such that $W=\{W(t): t \ge 0\}$ has stationary independent increments as in (\ref{indep incr}).

\begin{secthm}\label{Y stationary} Let $\gamma$ be a non-degenerate mean zero Gaussian measure on the Borel sets of $E$ with norm $q$, and assume $Y=\{Y(t): t \ge 0\}$ is the $E$-valued sample continuous $\gamma$-generated Ornstein-Uhlenbeck process. Then the following hold:\\
(a)  $Y$ is a stationary mean zero Gaussian process in the sense of Definition \ref{def gaussian} and Definition \ref{def stationary}.\\
(b) The probability measure $\mu$ induced by $\{Y(t): 0 \le t \le 1\}$ on $B=C_E[0,1]$ with norm $\|\cdot\|$ as in (\ref{norm C_{E}}) is a non-degenerate mean zero Gaussian measure in the sense that every $f \in B^*$ is a mean zero Gaussian random variable with variance as in (\ref{var f}). Moreover, the sample continuous processes $\{X_k(t): t \in [0,1]\}$ defined in (\ref{X_{k}}) are Gaussian in the sense of Definition \ref{def gaussian} and they induce mean zero Gaussian measures $\{\mu_k:k \ge 1\}$ on the Borel subsets of  
$B$ such that
\begin{align*}
\mathcal{L}(X_k)=\mu_k=\mu, k \ge 1.
\end{align*}
(c) If $K$ is the unit ball of the Hilbert space $H_{\mu}$ that generates $\mu$ and
\begin{align}\label{Gamma=sup}
\Gamma=\sup_{x \in K}\|x\|=\sup_{x \in K} \sup_{0 \le t \le 1}q(x(t)), 
\end{align}
then $\Gamma \in (0,\infty)$ and with probability one
\begin{align}\label{limsup Y le 0 continuous}
\limsup_{T \rightarrow \infty} [\sup_{0 \le t \le T} q(\frac{Y(t)}{\Gamma}) - \sqrt{2LT}]  \le 0, 
\end{align}
and 
\begin{align}\label{liminf Y ge 0}
\liminf_{T \rightarrow \infty} [\sup_{0 \le t \le T} q(\frac{Y(t)}{\Gamma}) - \sqrt{2LT}] \ge 0.
\end{align}
\end{secthm}

\begin{seccor}\label{gaussian D-E} Let $\gamma$ be a non-degenerate mean zero Gaussian measure on $E$, and assume  $G_1,G_2,\cdots$ are i.i.d. Gaussian random vectors with distribution $\gamma$. If $S_k=G_1+\cdots G_k$ for $k \ge 1$ and $\Gamma$ is as in (\ref{Gamma=sup}), then with probability one
\begin{align}\label{limsup normalized le 0}
\limsup_{ n \rightarrow \infty} [\max_{1 \le k \le n} q( \frac{S_k}{\sqrt{k}\Gamma}) - \sqrt{2LLn}]  \le 0. 
\end{align}
and
\begin{align}\label{liminf normalized ge 0}
\liminf_{ n \rightarrow \infty} [\max_{1 \le k \le n} q( \frac{S_k}{\sqrt{k}\Gamma}) - \sqrt{2LLn}]  \ge 0. 
\end{align}
\end{seccor}

\begin{proof}[Proof of Theorem 4.1] If $Y=:\{Y(t): t \ge 0\}$ is sample continuous and as in (\ref{E-valued OU}), then Lemmas \ref{OU gaussian} and \ref{OU stationary} show $Y$ is a stationary, mean-zero Gaussian process in the sense of Definitions \ref{def gaussian} and \ref{def stationary}. Hence (a) in Theorem \ref{Y stationary} holds. In addition, Lemmas \ref{f(X) gaussian} and \ref{Xk gaussian} are then applicable and imply that the sample continuous processes $\{X_k(t): t \in [0,1]\}$ defined in (\ref{X_{k}}) are Gaussian in the sense of Definition \ref{def gaussian}. Moreover, since $\gamma$ is assumed non-degenerate they show that the mean zero Gaussian measures $\{\mu_k:k \ge 1\}$ induced on the Borel subsets of  $B$ are also non-degenerate and such that
\begin{align*}
\mathcal{L}(X_k)=\mu_k=\mu, k \ge 1.
\end{align*}
It also follows from Lemma \ref{f(X) gaussian} that for every $f \in C_E^{*}[0,1]$, $f(X)$ is a mean zero Gaussian random variable with variance as in (\ref{var f}) with $B=C_E[0,1]$. Therefore, (b) also holds. 

To prove (c) we first observe that $\Gamma< \infty$ since $K$ is a compact subset of $C_E[0,1]$, and it is strictly positive since $\mu$ is non-degenerate when we assume $\gamma$ is non-degenerate, which  implies the unit ball $K$ of the Hilbert space $H_{\mu}$ is non-degenerate. Combining (b) and $\Gamma$ as in (\ref{Gamma=sup})
we now have (\ref{limsup Y le 0 continuous}) with probability one by (\ref{limsup Y le 0}) in Corollary \ref{main with OU}. Finally, (\ref{liminf Y ge 0}) holds with probability one from (\ref{liminf Y ge 0 corr}) in Corollary \ref{main with OU} since we can check (\ref{correl decay}).
That is, for $x_0 \in K$ such that $\|x_0\|=\Gamma$ the Hahn-Banach theorem implies there is a linear functional $f_0 \in B^*$ such that $f_0(x_0) =  \Gamma$, $\|f_0\|^{*}=1$ and Lemma \ref{link gamma sigma} implies
\begin{align*}
\sigma_{f_0}^{2} =\int_B f_0^{2}(x)d\mu(x) =\Gamma^2.
\end{align*}
Furthermore, Lemma \ref{f(Xk) stationary} implies the sequence of mean zero random variables $\{f_0(X_k): k \ge 1\}$ is stationary with variance $\Gamma^2$ and Lemma \ref{decay rate lemma} shows
\begin{align}
\lim_{n \rightarrow \infty} ( \log_{e}n)E[f_0(X_1)f_0(X_n)] = 0, 
\end{align}
which when combined with Corollary \ref{main with OU} completes the proof of Theorem \ref{Y stationary}.
\end{proof}

\begin{proof}[Proof of Corollary \ref{gaussian D-E}] Since $Y=\{Y(t):t \ge 0\}$ is as in (\ref{E-valued OU}) with $\{W(t): t \ge 0\}$ the $E$-valued sample continuous Brownian motion induced by $\gamma$, it follows that the sequences 
$\{G_k: k \ge\}$ and $\{W(k)-W(k-1): k \ge1\}$ with $W(0)=0$ have the same law. Therefore,  the $E^{\infty}$-valued random vectors
\begin{align}
(G_1,\frac{G_1+G_2}{\sqrt{2}},\cdots)~{\rm{and}}~(W(1),\frac{W(2)}{\sqrt{2}},\cdots) 
\end{align}
have the same law, and assuming without loss of generality that $G_k=W(k)-W(k-1)$ for all $k \ge 1$ we have $W(k)/\sqrt{k}= Y(Lk)$ with probability one for all $k \ge 1$. Hence with probability one for all $n\ge 1$ we have
\begin{align}
\max_{1 \le k \le n} q(\frac{G_1+\cdots+G_k}{\sqrt{k}\Gamma})=\max_{1 \le k \le n}q(\frac{ Y(Lk)}{\Gamma})
\end{align}
which implies with probability one that
\begin{align}
\limsup_{n \rightarrow \infty}[\max_{1 \le k \le n} &q(\frac{G_1+\cdots+G_k}{\sqrt{k}\Gamma}) - \sqrt{2LLn}]=\limsup_{n \rightarrow \infty}[\max_{1 \le k \le n} q(\frac{Y(Lk)}{\Gamma}) - \sqrt{2LLn}] \\
&\le \limsup_{n \rightarrow \infty}[\max_{1 \le t \le Ln} q(\frac{Y(t)}{\Gamma}) - \sqrt{2LLn}] \le 0,\notag
\end{align}
where the inequality follows immediately from (\ref{limsup Y le 0 continuous}) with $T=Ln$. Hence (\ref{limsup normalized le 0}) holds, and we now turn to (\ref{liminf normalized ge 0}).

Let $\Gamma$ be as in (\ref{Gamma=sup}) and $\Gamma_{\gamma}= \sup_{z \in K_{\gamma}}q(z)$, where $K_{\gamma}=\{z: \|z\|_{H_{\gamma}} \le 1\}$ is the unit ball of the Hilbert space $H_{\gamma}$.  Then Lemma \ref{link gamma sigma} implies $\Gamma= \sigma(\mu)$, where $\mu = \mathcal{L}(X_1)$,  and hence by (\ref{decay rate formula six}) we have
\begin{align*}
\Gamma \le \sigma(\gamma)= \Gamma_{\gamma}.
\end{align*}
Furthermore, by the conclusion of Remark \ref{Gamma=gamma} we have 
\begin{align}\label{Gamma=Gamma gamma}
\Gamma = \Gamma_{\gamma},
\end{align}
and applying Lemma \ref{link gamma sigma} to the mean zero Gaussian measure $\gamma$ on $E$, there exists $f_0 \in E^*$, $ \|f_0\|^*=1$, $ f_0(z_0)=\Gamma_{\gamma}$ ($z_0$ as above), such that
\begin{align}
\int_E f_0^2(z) d\gamma(z)= \Gamma_{\gamma}^2. 
\end{align}

Since $\Gamma = \Gamma_{\gamma}$ and $S_k =G_1+\cdots +G_k$ for $k\ge 1$,
\begin{align}\label{max q ge max f0}
\max_{1 \le k \le n}q(\frac{S_k}{\sqrt{k}\Gamma})\ge \max_{1 \le k \le n}\frac{|f_0(S_k)|}{\sqrt{k}\Gamma_{\gamma}}
\ge  \max_{0 \le j \le j_n}|Y_j|, 
\end{align}
where $ j_n= \max\{j: 2^{j} \le n\}$ for $n\ge 1$ and 
\begin{align*}
Y_j= \frac{f_0(G_1) + \cdots+ f_0(G_{2^{j}})}{\sqrt{2^{j}}\Gamma_{\gamma}}, j\ge 1.
\end{align*}
Therefore, $E(Y_j)=0, E(Y_j^2)=1$, and for $0\le i \le j < \infty$
\begin{align*}
E(Y_iY_j) = \frac{E(f_0^2(G_1)) 2^i}{\Gamma_{\gamma}^2\sqrt{2^{(i+j)}}}= 2^{-(j-i)/2},
\end{align*}
which implies  $\{Y_j: j \ge 0\}$ is a mean zero-variance one  stationary Gaussian sequence of real random variables with
\begin{align*}
r_k =E(Y_0Y_k) = e^{-k/2},
\end{align*}
and hence Theorem 3.3 in \cite{pickands} implies 
\begin{align}\label{liminf maxY ge 0}
\liminf_{k \rightarrow \infty}[\max_{0 \le j \le k} Y_j - \sqrt{2Lk}] \ge 0 
\end{align}
with probability one.

From (\ref{max q ge max f0}) we have
\begin{align}\label{liminf max normalized ge +epsilon n}
\liminf_{n \rightarrow \infty}[\max_{1 \le k \le n} q( \frac{S_k}{\sqrt{k}\Gamma})- \sqrt{2LLn}] \ge \liminf_{n \rightarrow \infty}[  \max_{0 \le j \le j_n} Y_j - \sqrt{2Lj_n} +\epsilon_n] 
\end{align}
where $n \rightarrow \infty$ implies
\begin{align*}
|\epsilon_n| = |\sqrt{2Lj_n} - \sqrt{2LLn}| \rightarrow 0.
\end{align*}
Therefore, (\ref{liminf maxY ge 0}) combined with (\ref{liminf max normalized ge +epsilon n}) implies  (\ref{liminf normalized ge 0}) with probability one, and the corollary is proved. 
\end{proof}

\subsection{An Application to Random operators}\label{random operators} 
Let $H$ be a separable Hilbert space over the complex numbers with inner product $\langle x, y \rangle, x,y \in H, $ and for a bounded operator $A$ from $H$ to $H$ denote the uniform operator norm by
\begin{align*}
q(A)=:\sup_{h \in H, \langle h, h \rangle =1} \langle Ah, Ah \rangle^{\frac{1}{2}}.
\end{align*}
In this sub-section we assume $E$ is a separable Banach space over the real numbers consisting of bounded self-adjoint operators 
with norm $q(\cdot)$, and that $\gamma$ is a mean zero Gaussian measure on the Borel subsets of $(E, q)$. The Hilbert space generating $\gamma$ will be denoted by $H_{\gamma}$, its unit ball by $K_{\gamma}$, and
\begin{align}\label{Gamma sigma=sup q}
\Gamma_{\gamma}=: \sup_{z \in K_{\gamma}} q(z).
\end{align}
Some standard facts about elements of $E$ are as follows. If $A$ is self-adjoint, then the spectrum of $A$, denoted by $\sigma(A)$, is a compact non-empty subset of real numbers with spectral radius 
\begin{align}
r_{\sigma}(A)=: \sup\{|\lambda|: \lambda \in \sigma(A)\}, 
\end{align}
and
\begin{align}\label{rsigma=q}
r_{\sigma}(A)= q(A). 
\end{align}
If $A$ is a compact self-adjoint operator from $H$ to $H$, then $\sigma(A)$ is a countable set of real numbers consisting of the eigenvalues of $A$ and
\begin{align*}
\sigma(A) \cap ((-\infty,\infty)- \{0\}) = \sigma(A)^{+} \cup \sigma(A)^{-},
\end{align*}
where $\sigma(A)^{+}$ denotes the strictly positive eigenvalues of $A$ and $\sigma(A)^{-}$ is the strictly negative eigenvalues of $A$. Furthermore, zero may or may not be an eigenvalue of $A$, but if $0 \notin \sigma(A)$ and $H$ is infinite dimensional, then it is always a limit point of either $\sigma(A)^{+}$ or $\sigma(A)^{-}$. 

Perhaps somewhat less well known are the following facts, and hence a detailed summary  appears in  the appendix \cite{partial-max-appendix}. 

For compact self-adjoint operators on the infinite dimensional Hilbert space $H$, let $\Sigma$ be the set-valued map on these operators defined by
\begin{align*} 
\Sigma(A) = \sigma(A),
\end{align*}
and define the Hausdorff metric  distance between $\sigma(A)$ and $ \sigma(B)$ by
\begin{align}\label{dsigmaAsigmaB}
d(\sigma(A), &\sigma(B))
=: \inf\{\delta>0: \sigma(A) \subseteq \sigma(B) +(-\delta,\delta)\\
& \textrm{and}\, \sigma(B) \subseteq \sigma(A) +(-\delta,\delta)\}. \notag
\end{align}
Then, for $A,B$ compact self-adjoint operators on $H$ it is known that
\begin{align}\label{q unif contin}
q(A-B)< \delta ~{\rm{implies}}~~d(\sigma(A), \sigma(B))<\delta. 
\end{align}

For a proof of (\ref{q unif contin}) see Theorem 3 in \cite{partial-max-appendix}. In particular, if $E$ consists of compact self-adjoint operators on $H$ and
\begin{align*}
E_{\sigma} = \{ \sigma(A): A \in E\}
\end{align*}
with distance on $E_{\sigma}$ the Hausdorff metric in (\ref{dsigmaAsigmaB}), then the map $\Sigma: A\rightarrow \sigma(A)$ is a Lip-1 continuous map from $(E,q)$ onto $(E_{\sigma}, d)$ since
(\ref{q unif contin}) implies
\begin{align*}
d(\Sigma(A), \Sigma(B)) \le 2q(A-B).
\end{align*}
Thus for $(E,q)$ a Banach space of compact self-adjoint operators and $\{A_n: n \ge 1\}$ a sequence in $E$ and $A \in E$ 

\begin{align}\label{q(A-An)}
\lim_{n \rightarrow \infty} q&(A-A_n)=0~{\rm {implies}}\\
&\lim_{n \rightarrow \infty} d(\Sigma(A),\Sigma(A_n))=\lim_{n \rightarrow \infty}d(\sigma(A), \sigma(A_n))= 0.  \notag
\end{align}
As usual in any metric space $(M,\rho)$, for $x \in D, D \subseteq M$, we define $\rho(x,D) = \inf_{a \in D} \rho(x,a)$ and the cluster set $C(\{x_n\})$ to be the set of all limit points of the sequence $\{x_n\} \subseteq M$ taken in $(M,\rho)$. Thus for $(E,q)$ a Banach space of compact self-adjoint operators, $\{A_n: n \ge 1\}$ a sequence in $E$, and $D \subseteq E$ we have that
\begin{align}\label{convergence of sigma}
\lim_{n \rightarrow \infty} &q(A_n,D)=0~{\rm {implies}}\\
& \lim_{n \rightarrow \infty} d(\Sigma(A_n),\Sigma(D))= \lim_{n \rightarrow \infty}\inf_{a \in D}d(\sigma(A_n), \sigma(a))= 0.  \notag
\end{align}
Moreover, if
\begin{align*}
E_{\sigma} = \Sigma(E)=\{ \sigma(A): A \in E\}
\end{align*}
with distance on $E_{\sigma}$ the Hausdorff metric in (\ref{dsigmaAsigmaB}), then for compact subsets $D$ of $(E,q)$
\begin{align}\label{qAn}
\lim_{n \rightarrow \infty} &q(A_n,D)=0~{\rm {and }}~~C(\{A_n\})=D &\ \textrm{implies}\\
&C(\{\Sigma(A_n)\})=\Sigma(D), \notag
\end{align}
where the cluster set $C(\{\Sigma(A_n)\})$ is computed relative to $(E_{\sigma},d)$.

From these standard facts it follows rather easily that our results have implications for the spectrums of random operators with centered Gaussian distribution $\gamma$ on $E$. As a sample we will present applications of \ref{lim Mn} in Corollary \ref{limsup le 0 f0 stat} and \ref{gaussian D-E}. The interested reader should then envision others. 

\begin{seccor} Let $E$ be a separable Banach space over the real numbers consisting of bounded self-adjoint operators on the Hilbert space $H$ in the uniform operator norm $q(\cdot)$, and that $\gamma$ is a mean zero Gaussian measure on the Borel subsets of $(E,q)$ with $ \Gamma_{\gamma}$ as in (\ref{Gamma sigma=sup q}). If $A_1,A_2, \dots $ are i.i.d. $E$-valued random operators with  law $\gamma$ and  $S_k=A_1+\cdots A_k$ for $k \ge 1$, then with probability one 
\begin{align}\label{lim max rsigma}
\lim_{n \rightarrow \infty} [\max_{1 \le k \le n} \frac{r_{\sigma}(A_k)}{\Gamma_{\gamma}} -\sqrt{2Ln}] =0 
\end{align}
and
\begin{align}\label{lim max normalized rsigma}
\lim_{ n \rightarrow \infty} [\max_{1 \le k \le n} \frac{r_{\sigma}(S_k)}{\sqrt{k}\Gamma_{\gamma}} - \sqrt{2LLn}]  = 0. 
\end{align}
In addition, if we assume the operators are compact and $d(\cdot,\cdot)$ is the Hausdorff metric on $E_{\sigma}$, then with probability one 
\begin{align}\label{d(sigma)}
\lim_{n \rightarrow \infty}d(\frac{\sigma(A_n)}{\sqrt{2Ln}}, \mathcal{K})=0, 
\end{align}
and
\begin{align}\label{C(sigma)=}
C(\{\frac{\sigma(A_n)}{\sqrt{2Ln}}\} ) = \mathcal{K}, 
\end{align}
where
\begin{align}\label{mathcal K=}
\mathcal{K}=: \Sigma(K_{\gamma})=\{\sigma(A): A \in K_{\gamma}\}. 
\end{align}
With probability one we also have
\begin{align}\label{distance normalized mathcal K}
\lim_{n \rightarrow \infty}d(\frac{\sigma(S_n)}{\sqrt{2nLLn}}, \mathcal{K})=0, 
\end{align}
and
\begin{align}\label{C=mathcal K}
C(\{\frac{\sigma(S_n)}{\sqrt{2nLLn}}\} ) = \mathcal{K}. 
\end{align}
\end{seccor}

\begin{proof} The conclusion in (\ref{lim max rsigma}) follows immediately from (\ref{lim Mn}) and (\ref{rsigma=q}). Similarly, (\ref{lim max normalized rsigma}) follows from Corollary \ref{gaussian D-E}, (\ref{rsigma=q}), and that $r_{\sigma}(\frac{S_k}{\sqrt{k}})= \frac{r_{\sigma}(S_k)}{\sqrt{k}}$. 

To verify (\ref{d(sigma)})\, and (\ref{C(sigma)=}) we observe from Theorem 1 in \cite{goodman-kuelbs-weakly} or Theorem 2.1 in \cite{goodman-kuelbs-clustering-ss} for every $\epsilon>0$ and $\epsilon_n=\epsilon/\sqrt{2Ln}$ that with probability one
\begin{align}\label{K gamma +epsilon n eventually}
&P( \frac{A_n}{\sqrt{2Ln}} \in K_{\gamma}+\epsilon_nU~ {\rm{eventually}}) =1, 
\end{align}
where $U=\{A\in E:q(A)<1\}$. Thus with probability one we have $ \frac{A_n}{\sqrt{2Ln}}$ converging to $K_{\gamma}$ in $(E,q)$, and
\begin{align}\label{q K gamma eventually}
P(q( \frac{A_n}{\sqrt{2Ln}},K_{\gamma}) < \epsilon_n~{\rm { eventually}})=1.
\end{align}
Since $\Sigma$ is Lip-1 as above and (\ref{mathcal K=}) holds, then

\begin{align*}
\{d( \frac{\Sigma(A_n)}{\sqrt{2Ln}},\mathcal{K}) < 2\epsilon_n~{\rm { eventually}}\} \supseteq \{q(\frac{A_n}{\sqrt{2Ln}}, K_{\gamma}) < \epsilon_n~{\rm{eventually}} \},
\end{align*}
and hence we have (\ref{d(sigma)}) with probability one.

To verify (\ref{C(sigma)=}) we first observe Th\'eor\`eme 4.1 of \cite{carmona-kono} implies
\begin{align}\label{C= K gamma=1}
P(C(\{\frac{A_n}{\sqrt{2Ln}}\} ) = K_{\gamma})=1, 
\end{align}
and since $K_{\gamma}$ is compact in $E$, (\ref{q K gamma eventually}) and (\ref{C= K gamma=1})  combined with (\ref{q(A-An)}) gives (\ref{C(sigma)=})  with probability one. 

Finally, the proof of (\ref{distance normalized mathcal K}) and (\ref{C=mathcal K}) follow as that for (\ref{d(sigma)}) and (\ref{C(sigma)=}) since Theorem 4.1 in \cite{lil-banach-special} implies the law of the iterated logarithm for the i.i.d. centered $E$-valued Gaussian  random vectors $\{A_k: k \ge 1\}$, and hence
\begin{align}\label{q An=1}
P(\lim_{n \rightarrow \infty}q( \frac{S_n}{\sqrt{2nLLn}},K_{\gamma}) =0) = P(C(\{\frac{S_n}{\sqrt{2nLLn}}\} ) = K_{\gamma})=1.
\end{align}
Thus the corollary is proved.
\end{proof}

\begin{secrem}\label{loglog implications} Since $\epsilon_n =\epsilon/\sqrt{2Ln}$ the statements in 
(4.32-33)
imply 
\begin{align*}
P(\lim_{n \rightarrow \infty} q(A_n, \sqrt{2Ln}K_{\gamma})=0)=1~{\rm{and}}~P( \lim_{n \rightarrow \infty}d( \Sigma(A_n), \sqrt{2Ln} \mathcal{K}) =0)=1.
\end{align*}
\end{secrem}

\subsection{An Application to Medians for Gaussian Maximal Functions}\label{median centerings} 
Next we establish a symmetrization result, and indicate how it provides precise asymptotic behavior for the medians of the partial maxima of Gaussian vectors when combined with our results above.

\begin{prop}\label{centering}Assume $x_{n}$ is a real-valued random variable with unique median, $m_{n}$, and $x_{n}^{'}$ is an independent copy of $x_{n}$. Then,
\begin{align}\label{probab}
&x_{n}-x_{n}^{'}\to 0\  \text{\rm in probability implies } x_{n}-m_{n}\to 0 \ \text{\rm in probability, and}
\end{align}
 \begin{align}\label{a.s.convergence}
&x_{n}-x_{n}^{'}\to 0\ \text{\rm a.s. implies }\, x_{n}-m_{n}\to 0\ \text{\rm a.s.}
\end{align}

\end{prop}

\begin{proof}We'll use $2\Pr( x_{n}^{'}\le m_{n})=1$ and $2\Pr( x_{n}^{'}\ge m_{n})=1$. Fix $\epsilon>0$.
\begin{align*}\Pr(x_{n}&>m_{n}+\epsilon)=\big[2\Pr( x_{n}^{'}\le m_{n})\big] \Pr(x_{n}>m_{n}+\epsilon)\\
&= 2\Pr(x_{n}>m_{n}+\epsilon, x_{n}^{'}\le m_{n})\le 2\Pr(x_{n}>x_{n}^{'}+\epsilon, x_{n}^{'}\le m_{n})
\notag\\
&\le 2\Pr(x_{n}>x_{n}^{'}+\epsilon)\to 0.
\notag
\end{align*}
Similarly, 
\begin{align*}\Pr(x_{n}&<m_{n}-\epsilon)=\big[2\Pr( x_{n}^{'}\ge m_{n})\big] \Pr(x_{n}<m_{n}-\epsilon)\\
&= 2\Pr(x_{n}<m_{n}-\epsilon, x_{n}^{'}\ge m_{n})\le 2\Pr(x_{n}< x_{n}^{'}-\epsilon, x_{n}^{'}\ge m_{n})
\notag\\
&\le 2\Pr(x_{n}<x_{n}^{'}-\epsilon)\to 0.
\notag
\end{align*}
Hence ( \ref{probab}) holds.

Now assume $x_{n}-x_{n}^{'}\to 0$, a.s. Then, 
by Fubini's theorem there is a particular ``$\omega^{'}$'' such that $x_{n}-x_{n}(\omega^{'})\to 0$, a.s. and letting $b_{n}=x_{n}(\omega^{'})$ we have, $x_{n}-b_{n}\to 0$ in probability. Since, the assumption also implies $x_{n}-x_{n}^{'}\to 0$ in probability, we have by (\ref{probab}) that $x_{n}-m_{n}\to 0$  in probability. Therefore, $b_{n}-m_{n}\to 0$ in probability. Since these last are constants, we have $b_{n}-m_{n}\to 0$. Going back to $x_{n}-b_{n}\to 0$ a.s., we get $x_{n}-m_{n}\to 0$, a.s. 
\end{proof}

\begin{prop}\label{medians via symm} Let  $\mu, \mu_1,\mu_2,\cdots$ be  centered non-degenerate Gaussian measures on 
$B$ with norm $\|\cdot\|$, and $X,X_1,X_2,\cdots$ be $B$-valued random vectors on some probability space with distributions $\mu, \mu_1,\mu_2,\cdots$ such that $\{\mu_n: n \ge 1\}$ converges weakly to $\mu$ on $B$. In addition, assume $\{b_n\}$ are numbers such that
\begin{align}\label{convergence assumed}
\lim_{n \rightarrow \infty}[M_n - b_n] = 0
\end{align}
where the convergence to zero is in probability, and
\begin{align*}
\Gamma= \sup_{x\in K} \|x\|~{\rm{and}}~M_n=: \max_{1 \le k \le n}\frac{\|X_k\|}{\Gamma}.
\end{align*}
Then, 
\begin{align}\label{median precision}
\lim_{n \rightarrow \infty}[med(M_n)-b_n]=0.
\end{align}
\end{prop}

The proof of Proposition \ref{medians via symm} is immediate since the distribution function of $M_n$ is continuous and strictly increasing on $(0,\infty)$ and zero elsewhere when $B$ is separable and the Gaussian random vectors are assumed to be centered and non-degenerate. Since results in the literature as well as in Sections 2,3,4(above) imply (\ref{convergence assumed}) with $b_n$  either $\sqrt{2Ln}$ or $\sqrt{2LLn}$,  the asymptotic behavior of the medians of $M_n$ in these situations are precisely determined by (\ref{median precision}). Furthermore, determining precise asymptotic behavior for these medians by direct calculation does not appear to be immediate.

\section{Properties of the Banach-valued OU 
Process}\label{OU} Here we present some results on sample function continuity and stationarity properties for the $E$-valued Ornstein-Uhlenbeck process of (\ref{E-valued OU}) that were used in the proofs in Section \ref{banach OU}. We start with the sample function continuity of the related Brownian motion in Lemma \ref{existence of brownian}, which appeared earlier in \cite{gross-potential}. In view of this, and the additional results and proofs in \cite{partial-max-appendix}, we  only state the necessary facts and related notation here. That the resulting $\gamma$-Brownian motion is a mean zero Gaussian process in the sense of Definition \ref{def gaussian} can easily be seen from the proof of Lemma
\ref{OU gaussian} below.

\subsection{Sample Function Continuity}\label{sample function 1} 
\numberwithin{equation}{subsection}
As before assume $\gamma$ is a mean zero Gaussian measure on $E$ with norm $q$. In addition, assume $\Omega_E$ is the space of continuous functions $x$ from $[0,\infty)$ into $E$ such that $x(0)=0$, and $\mathcal{F}$ is the $\sigma$-field of $\Omega_E$ generated by the functions $x \rightarrow x(t), 0 \le t < \infty$. Our next lemma provides a proof that there exists a probability measure $P_{\gamma}$ on $(\Omega_E,\mathcal{F})$ such that if $0=t_0<t_1\cdots< t_n$, then the random vectors
\begin{align}\label{indep incr}
x(t_j)-  x(t_{j-1}), j=1,\cdots ,n, ~{\rm{are~independent}}, 
\end{align}
and $x(t_j)- x(t_{j-1})$ has distribution $\gamma_{t_j - t_{j-1}}$ on $E$, where $\gamma_{s}(A)= \gamma(A/\sqrt{s})$ for Borel subsets $A$ of $E$ when $s>0$ and $\gamma_0=\delta_0.$ In particular, the stochastic process  $\{W(t): t\ge 0\}$ defined on $(\Omega_E, \mathcal{F},P_{\gamma})$ by $W(t,x)=x(t)$ has stationary independent mean zero Gaussian increments. We will call it $\gamma$-Brownian motion on $E$. 

\begin{subseclem}\label{existence of brownian}
Let $\gamma$ be a mean zero Gaussian measure on $E$. 
Then, the $E$-valued Brownian motion $W=\{W(t): t\ge 0\}$ defined on $(\Omega_E, \mathcal{F},P_{\gamma})$ by $W(t,x)= x(t)$ exists. In particular, it is sample path continuous and a mean zero Gaussian process (in the sense of Definition \ref{def gaussian}) with\\
\indent {(i) $W(0)=0,$ }\\
\indent {(ii) $W$ has stationary independent mean zero Gaussian increments as indicated above, and}\\
\indent {(iii) if the support of $\gamma$} is a closed subspace $F$ of $E$, then the probability $P_{\gamma}$ has support on $\Omega_F$, the space of continuous functions on $[0,\infty)$ with values in $F$, and $W$ is a $\gamma$ Brownian motion on F.\\
\end{subseclem}

\begin{subsecrem} It is usual to assume the support of $\gamma$ to be $E$ when we define $\gamma$-Brownian motion, but in later proofs it will be convenient to keep it mind this need not be the case.
\end{subsecrem}

\subsection{Gaussian and Stationary Properties for Ornstein-Uhlenbeck Processes}\label{existence OU}
Let $Y$ be a sample continuous Ornstein-Uhlenbeck process generated by $\gamma$ as in (\ref{E-valued OU}). In this subsection we then prove $Y$ is a mean zero Gaussian process in the sense of Definition \ref{def gaussian} and also a stationary process in the sense of Definition \ref{def stationary}. As a result, it follows that both Lemma \ref{f(X) gaussian} and \ref{Xk gaussian} apply to $Y$.

\begin{subseclem}\label{OU gaussian}The sample continuous Ornstein-Uhlenbeck process$ \{Y(t): t \ge 0\}$ is a mean zero Gaussian process in the sense of Definition \ref{def gaussian}.
\end{subseclem}

\begin{proof} For each integer $d\geq 1$ and 
$0 \le t_1<t_2,<\cdots <t_d$  we need to show the finite dimensional distribution of
\begin{align}
(Y(t_1),\cdots,Y(t_d))  
\end{align}
is a mean zero Gaussian measure on $E^d$. Since the typical continuous linear functional on $E^d$ is of the form
\begin{align}
f(x_1,\cdots,x_d) = \sum_{j=1}^d \theta_j(x_j), 
\end{align} 
where $\theta_1, \cdots,\theta_d \in E^*$, and $E^d$ is a separable Banach space, it suffices to show
\begin{align}
\sum_{j=1}^d \theta_j(Y(t_j)) 
\end{align}
is a mean zero Gaussian random variable.

If $d=1$, then
\begin{align*}
\theta_1(Y(t_1)) =e^{-\frac{t_1}{2}}\theta_1(W(e^{t_1})),
\end{align*}
and hence $Y(t_1)$ is a mean zero Gaussian random variable since
$\theta_1(W(s))$ is a mean zero Gaussian random variable for each $s \ge 0$. Moreover, by the scaling property of the $E$-valued Brownian motion
\begin{align*}
E[\theta_1^{2}(Y(t_1))] = E[\theta_1^{2}(W(1))] .
\end{align*}
To handle the situation for $d\ge 2$ it is useful to consider the following identity which is easily established by induction. That is, for $a_j, 1 \le j \le d,$ linear functionals on $E$ and $b_j, 0 \le j \le d$ points in $E$ with $b_0=0$ we have
\begin{align*}
\sum_{j=1}^d a_j( b_j)= \sum_{j=1}^d \sum_{k=j}^d a_k(b_j -b_{j-1})=  \sum_{j=1}^d \big(\sum_{k=j}^d a_k\big)(b_j -b_{j-1}). 
\end{align*}
Therefore, with $a_j= e^{-\frac{t_j}{2}} \theta_j$  and $b_j=W(e^{t_j})$ for $1 \le j \le d$ with $b_0=W(e^{t_0})=0$ we have 
\begin{align}\label{sum theta j}
\sum_{j=1}^d \theta_j(Y(t_j))= \sum_{j=1}^d  e^{-\frac{t_j}{2}}\theta_j(W(e^{t_j}))= \sum_{j=1}^d \big(\sum_{k=j}^d  e^{-\frac{t_k}{2}}\theta_k\big)(W(e^{t_j}) -W(e^{t_{j-1}})). 
\end{align}
Hence, by  the independent increments of the $\gamma$-Brownian motion on $E$ we thus have $\sum_{j=1}^d \theta_j(Y(t_j))$ is the sum of the independent mean zero Gaussian random variables
\begin{align}\label{substitute W}
\big(\sum_{k=j}^d  e^{-\frac{t_k}{2}}\theta_k\big)(W(e^{t_j}) -W(e^{t_{j-1}})), 1 \le j \le d, 
\end{align} 
which completes the proof.
\end{proof}

\begin{subseclem}\label{OU stationary} The sample continuous Ornstein-Uhlenbeck process$ \{Y(t): t \ge 0\}$ is a stationary mean zero Gaussian process in the sense of Definitions \ref{def gaussian} and \ref{def stationary}.
\end{subseclem}

\begin{proof} Given Lemma (\ref{OU gaussian}) it suffices to show for each integer $d\geq 1$, $h>0$, and 
$0 \le t_1<t_2,<\cdots <t_d$  that the finite dimensional distributions of 
\begin{align}
(Y(t_1),\cdots,Y(t_d)) ~{\rm{and}}~(Y(t_1+h),\cdots,Y(t_d+h)) 
\end{align}
are equal Gaussian probability measures on $E^d$.

By (\ref{OU gaussian}) we know both are mean zero Gaussian distributions on $E^d$, and hence they will be equal
if the variance of every linear functional on $E^d$ is the same for each distribution. That is, since the typical continuous linear functional on $E^d$ is of the form
\begin{align}
f(x_1,\cdots,x_d) = \sum_{j=1}^d \theta_j(x_j),
\end{align} 
where $\theta_, \cdots,\theta_d \in E^*$, and $E^d$ is a separable Banach space, it suffices to show
\begin{align}
E[(\sum_{j=1}^d \theta_j(Y(t_j)))^2] = E[(\sum_{j=1}^d \theta_j(Y(t_j+h)))^2]. 
\end{align}
Hence, from (\ref{sum theta j}), (\ref{substitute W}),  and
\begin{align*}
\psi_{j}= \sum_{k=j}^d e^{-\frac{t_k}{2}}\theta_k
\end{align*}
for $1 \le j \le d$, it suffices to show 
\begin{align}\label{variance increments}
E\big[  \{ \psi(W(e^{t_j}) -W(e^{t_{j-1}}))\}^{2} \big]= E\big[\{   e^{-\frac{h}{2}}\psi(W(e^{t_j+h}) -W(e^{t_{j-1}+h}))\}^{2}\big]. 
\end{align}
Using the scaling property and independent increments of the $E$-valued Brownian motion $\{W(t): t \ge 0\}$ we have
\begin{align*}
E\big[  \{ \psi(W(e^{t_j}) -W(e^{t_{j-1}}))\}^{2} \big] = (e^{t_j} - e^{t_{j-1}})E[\psi^2(W(1))],
\end{align*}
and
\begin{align*}
 E\big[\{   e^{-\frac{h}{2}}\psi(W(e^{t_j+h}) -W(e^{t_{j-1}+h}))\}^{2}\big]= e^{-h} (e^{t_j+h} - e^{t_{j-1}+h})E[\psi^2(W(1))],
\end{align*}
Thus (\ref{variance increments}) holds for $1\le j \le d$, which completes the proof.
\end{proof}

\subsection{More Stationary and Gaussian Properties for the Ornstein-Uhlenbeck Process}\label{Xk stationary} Let $\gamma$ be a non-degenerate mean zero Gaussian measure on the Borel sets of $E$, and assume 
\begin{align}\label{def Y from W}
Y(t)=: e^{-\frac{t}{2}}W(e^{t}), t \ge 0,
\end{align}
is the $E$-valued sample path continuous Ornstein-Uhlenbeck process generated by $\gamma$ normalized so that the law of $W(1)$, and hence also $Y(1)$, is $\gamma$ ( see Section \ref{sample function 1} for more details). In addition, assume the $C_E[0,1]$ valued random vectors $\{X_k: k \ge 1\}$ are defined as in (\ref{X_{k}}).
 
The next lemma shows that for every $f \in C_E^{*}[0,1]$ the sequence $\{f(X_k(\cdot): k\ge 1\}$ is a stationary sequence of real-valued mean zero Gaussian random variables.
Our proof depends on a Riesz representation result for such linear functionals established by Bochner and Taylor in \cite{bochner-taylor}, and for that we need some additional notation.

Define $V(E^*,q^*)$ to be the functions $g$ mapping $[0,1]$ into $E^*$ such that the sums
\begin{align*}
\sum_{j=1}^n q^* (g(t_j) - g(t_{j-1}))
\end{align*}
are uniformly bounded for all partitions $\{t_j\}$, where $0=t_0<t_1<\cdots<t_n=1$. For $g \in V(E^*,q^*)$, the least upper bound of all such sums is the $E^*$ bounded variation of $g$ and is denoted by $V(g,E^*)$.

If $x \in C_E[0,1]$ and $g \in V(E^*,q^*),$ Bochner and Taylor define the integral
\begin{align}\label{BT integral 1}
\int_0^1dg(t)x(t) 
\end{align}
in the usual way. That is, with $g$ fixed, for each $x \in C_E[0,1]$ and  each partition $D=\{t_j\}$ with $0=t_0<t_1<\cdots<t_n=1$ and points $\tau_j,~ t_{j-1} \le \tau_j \le t_j,$ denote the real-valued sum 
\begin{align}\label{BT integral 2}
S(D,x)=\sum_{j=1}^n [g(t_j) - g(t_{j-1})]x(\tau_j). 
\end{align}
If we put $|D| =\max_{1 \le j \le n} |t_j - t_{j-1}|$, then as indicated in \cite{bochner-taylor} in the usual way for a sequence of partitions $D_i$ with $|D_i| \rightarrow 0$ the sums $S(D_i,x)$ approach a limit (in $\mathbb{R}$) which is independent of the particular sequence $\{D_i\}$. This limit is called the integral of $x(\cdot)$ with respect to $g(\cdot)$. It is written as in
(\ref{BT integral 1}), and it follows immediately that
\begin{align}
|\int_0^1dg(t)x(t)| \le \sup_{0 \le t \le 1} q(x(t)) V(g,E^*).
\end{align}

The Riesz representation given by Bochner and Taylor is the following theorem. It will be used in Lemma \ref{f(Xk) stationary} below.

\begin{subsecthm}\cite{bochner-taylor} If $f $ is a continuous linear functional on $C_E[0,1]$, then there exists a function $g$ of $E^*$-bounded variation  $(g \in V(E^*,q^*))$ such that for each $x \in C_E[0,1]$
\begin{align}\label{integral from subsec thm}
f(x) = \int_0^1dg(t)x(t), 
\end{align}
where the integral is of the Riemann-Stieljes type as given in (\ref{BT integral 1})-(\ref{BT integral 2}), and 
\begin{align}\label{V(g,E*)}
\sup_{\{x \in C_E[0,1]: \|x\|=1\}}|f(x)| = V(g,E^*).  
\end{align}
\end{subsecthm}

\begin{subseclem}\label{f(Xk) stationary} Let $Y=:\{Y(t): t \ge 0\}$ be the sample continuous Ornstein-Uhlenbeck process generated by the non-degenerate mean zero Gaussian measure $\gamma$ on $E$ as in (\ref{def Y from W}), and assume the processes $\{X_k(t): 0 \le t \le 1\}$ are as in (\ref{X_{k}}). Then, for every $f \in C_E^{*}[0,1]$ the sequence $\{f(X_k(\cdot): k\ge 1\}$ is a stationary sequence of real-valued mean zero Gaussian random variables.
\end{subseclem}

\begin{proof} Given Lemmas \ref{OU gaussian} and \ref{OU stationary}, Lemma \ref{Xk gaussian} implies $\mu_k=\mathcal{L}(X_k),  k \ge 1,$ are identical mean zero Gaussian measures on $C_E[0,1]$, and the random variables $f(X_k), k \ge 1$, are mean zero Gaussian random variables with the same distribution. What remains is to show they are stationary and jointly Gaussian. 

Applying the Bochner-Taylor result, for each $f \in C_E^{*}[0,1]$ there exists  $g$ of $E^*$ bounded variation such that $f$ is as in (\ref{integral from subsec thm}), with  (\ref{V(g,E*)}) holding. Hence, for a sequence of partitions $\{D_i\}$ with $|D_i| \rightarrow 0,$ for all $x \in C_E[0,1]$
\begin{align}\label{lim for BT integral}
\lim_{i \rightarrow \infty} S(D_i,x) = f(x), 
\end{align}
where
\begin{align}\label{fi(x)}
f_i(x)=: S(D_i,x)= \sum_{j=1}^{r} \phi_j(x(\tau_j)), 
\end{align}
$\phi_j=:g(t_j) - g(t_{j-1}) \in E^*$, and $r+1$ equals the number of points in $D_i$. Of course, $r$ depends on $i$, but this is omitted to simplify the notation, and for all $x \in C_E[0,1]$ (\ref{lim for BT integral}) becomes 
\begin{align}\label{fi to f}
\lim_{i \rightarrow \infty} |f(x) - f_i(x)|=0. 
\end{align}

To show $\{f(X_k): k \ge 1\}$ is stationary and Gaussian we take $1 \le k_1<\cdots<k_d<\infty$ and $d\ge 1$ integers, and show 
\begin{align}\label{if show stationary}
\mathcal{L}(f(X_{k_1}),\cdots,f(X_{k_{d}}) )= \mathcal{L} (f(X_{k_1+h}),\cdots,f(X_{k_{d}+h}) ) 
\end{align}
are equal Gaussian probabilities on $\mathbb{R}^d$. Since (\ref{fi(x)}) and (\ref{fi to f}) hold, (\ref{if show stationary})  follows when we verify
\begin{align*}
\mathcal{L}(\sum_{j=1}^r\phi_j(X_{k_1}(\tau_j)&,\cdots,\sum_{j=1}^r \phi_j(X_{k_{d}}(\tau_j)) )\\
&=\mathcal{L}(\sum_{j=1}^r\phi_j(X_{k_1+h}(\tau_j),\cdots,\sum_{j=1}^r \phi_j(X_{k_{d}+h}(\tau_j)) ) \notag
\end{align*}
and that they are Gaussian measures.
Since $X_k(t)= Y(t +(k-1))$ for all $t \in [0,1]$ and integers $k \ge 1$, it suffices to show the probability measures
\begin{align}\label{first joint law}
\mathcal{L}\big(\sum_{j=1}^r\phi_j(Y(k_{1} -1+\tau_j)),\cdots,\sum_{j=1}^r \phi_j(Y(k_{d}-1 +\tau_j))\big)  
\end{align}
and 
\begin{align}\label{second joint law}
\mathcal{L}\big(\sum_{j=1}^r\phi_j(Y(k_{1}+h -1+\tau_j)),\cdots,\sum_{j=1}^r \phi_j(Y(k_{d}+h-1 +\tau_j))\big)  
\end{align}
are equal and jointly Gaussian. However, since
\begin{align*}
0 \le k_{1} -1+\tau_1<\cdots <k_1&-1+\tau_r \le k_{2}-1<\cdots <k_{2}-1+\tau_r \\
&\le \cdots \le k_{d}-1<\cdots <k_{d}-1+\tau_r\notag
\end{align*}
and $\{Y(t): t \ge 0\}$ is stationary and Gaussian we have
\begin{align*}
\mathcal{L}\big(Y(k_1 -1 +\tau_1),&\cdots, Y(k_1 -1 +\tau_r), \cdots,Y(k_d -1 +\tau_1),\\
&\cdots, Y(k_d -1 +\tau_r)\big),\notag
\end{align*}
and 
\begin{align*}
\mathcal{L} \big(Y(k_1+h -1 +\tau_1),&\cdots, Y(k_1+h -1 +\tau_r), \cdots,Y(k_d +h -1 +\tau_1),\\
&\cdots, Y(k_d +h -1 +\tau_r)\big),\notag
\end{align*}
are equal Gaussian probabilities on $\mathbb{E}^{rd}.$ From this, and that the $\{\phi_j: 1 \le j \le r\}$ are in $E^*$,  we immediately have the probabilities in (\ref{first joint law}) and (\ref{second joint law}) are equal and Gaussian, which completes the proof of the lemma.
\end{proof}

\subsection{Covariance Decay for the Ornstein-Uhlenbeck Process}\label{section covariance decay} 

In our proof of (\ref{liminf Y ge 0}) and (\ref{liminf normalized ge 0}) we have used Theorem 3.3 in \cite{pickands} which requires a rate of decay of the covariances for a sequence of stationary Gaussian random variables. Here we show that for the Ornstein-Uhlenbeck process the required rate of decay follows from the process itself, and is not an extra assumption. To estimate this covariance decay for the Ornstein-Uhlenbeck process generated by the mean zero Gaussian measure $\gamma$ on $E$ we first show the problem can be moved to any Banach space $F$ that is linearly isometric to $E$. Then, with suitably chosen $F$ we make the estimate. First, however we indicate some additional notation.

Let $E$ and $F$ be separable Banach spaces with norms $q_{E}$ and $q_{F}$ and suppose $\Lambda$: $E \rightarrow F$ is a linear isometry from $E$ onto $F$ with inverse $\Lambda^{-1}$. If $x \in C_E[0,\infty)$,  then we define the function $y(\cdot)={\hat{\Lambda}} x(\cdot) \in C_F[0,\infty)$ by
\begin{align}\label{hat lambda}
\hat{\Lambda}x(t)= \Lambda(x(t)), 0 \le t <\infty, 
\end{align}
and for  $x \in C_E[0,b]$ and $b>0$ we define 
\begin{align}
{\hat{\Lambda}}_{b} x(\cdot)= \Lambda (x(\cdot)), 0 \le t \le b. 
\end{align}
For every $b>0$ and $x \in C_E[0,\infty)$ we have for $y = \hat{\Lambda}(x)$ that
\begin{align}\label{hat lambda b}
\sup_{0 \le t \le b}q_{F}(y(t)) = \sup_{0 \le t \le b}q_{F}(\hat{\Lambda}(x)(t) = \sup_{0 \le t \le b}q_{F}(\Lambda(x(t))=\sup_{0 \le t \le b} q_{E}(x(t)),  
\end{align}
so $\hat {\Lambda}_b$ is a linear isometry from $ C_E[0,b]$ onto $C_F[0,b]$. Similarly,  $\hat \Lambda$ and $\hat{\Lambda}_b $ have linear inverses $\hat {\Lambda}^{-1}$ and $\hat{\Lambda}_{b} ^{-1} $ from  $C_F [0,\infty)$ onto $C_E[0,\infty)$ and $C_F[0,b]$ onto $C_E[0,b]$. Of course, $\Lambda^{-1}_{b}$ is also an isometry.
\bigskip

As before $\gamma$ is a non-degenerate mean zero Gaussian measure on the Borel subsets of $E$ and we define (as usual)
\begin{align}
\gamma^{\Lambda}(A) = \gamma(\Lambda^{-1}(A)). 
\end{align}
Then,  the probability $\gamma^{\Lambda}$ is a mean zero Gaussian measure on the Borel subsets of $F$, the supports of $\gamma$ and $\gamma^{\Lambda}$ are closed linear subspaces of $E$ and $F$, respectively, and they are linearly isometric under $\Lambda$. 
\bigskip

The $E$-valued $\gamma$-Brownian motion process on $(\Omega_{E},\mathcal{F},P_{\gamma})$ is defined as in Lemma \ref{existence of brownian}, and is indicated by
$W_{\gamma}=\{W_{\gamma}(t): t \ge 0\}$ with $W_{\gamma}(t,x)=x(t), x \in \Omega_{E}$.  The $E$-valued $\gamma$-Ornstein-Uhlenbeck process is given by
\begin{align}\label{gamma OU}
Y_{\gamma}(t,x)=: e^{-\frac{t}{2}}W_{\gamma}(e^{t},x)=  e^{-\frac{t}{2}}x(e^{t}), t \ge 0, x \in \Omega_E \subseteq C_E[0,\infty), 
\end{align}
and since the support of $\gamma$ is a closed subspace of E, the processes $W_{\gamma}$ and $Y_{\gamma}$ have support on the continuous functions on $[0,\infty)$ with values in that subspace.

Similarly, we let  $\Omega_F$ be the space of continuous functions $y$ from $[0,\infty)$ into $F$ such that $y(0)=0$, and $\mathcal{G}$ the $\sigma$-field of $\Omega_F$ generated by the functions $y \rightarrow y(t), 0 \le t < \infty$. Then, there is a probability $P_{\gamma^{\Lambda}}$ on $(\Omega_F,\mathcal{G})$ such that
the stochastic process  $W_{\gamma^{\Lambda}}=\{W_{\gamma^{\Lambda}}(t): t\ge 0\}$ defined on $(\Omega_F, \mathcal{G},P_{\gamma^{\Lambda}})$ by $W_{\gamma^{\Lambda}}(t,y)=y(t)$ is the $\gamma^{\Lambda}$-Brownian motion on $F$, and the stochastic process
\begin{align}\label{gamma hat lambda OU}
Y_{\gamma^{\Lambda}}(t,y)=: e^{-\frac{t}{2}}W_{\gamma^{\Lambda}}(e^{t},y)=  e^{-\frac{t}{2}}y(e^{t}), t \ge 0, y \in \Omega_F \subseteq C_F[0,\infty), 
\end{align}
is the $F$-valued $\gamma^{\Lambda}$-Ornstein-Uhlenbeck process with support the space of continuous functions on $[0,\infty)$ with values in the subspace of $F$ supporting $\gamma^{\Lambda}$.

{\bf Notation.} For $k \ge 1$ and $0 \le t \le 1$ we define the linear maps 
\begin{align}
\tau_k: \Omega_{E} \subseteq C_E[0,\infty) \rightarrow C_E[0,1], ~{\rm{and}}~{\hat{\tau}}_k: \Omega_{F} \subseteq C_F[0,\infty) \rightarrow C_F[0,1], 
\end{align}
by
\begin{align}\label{tau k and hat tau k}
\tau_k(x)(t)= e^{-\frac{t +(k-1)}{2}}x(e^{t+(k-1)})  ~ {\rm{and}}~\hat{\tau}_k(y)(t)= e^{-\frac{t +(k-1)}{2}}y(e^{t+(k-1)}),   
\end{align}
and the processes 
\begin{align}\label{x k and hat x k}
X_k(t,x)=\tau_k(x)(t), x \in \Omega_E, ~ {\rm{and}}~ {\hat{X}}_k(t,y)={\hat { \tau}}_k(y)(t), y \in \Omega_F. 
\end{align}
These maps are not only linear, but are continuous with respect to uniform convergence on compact subsets of $[0,\infty)$.

\begin{subseclem}\label{change variables}  $ P_{\gamma^{\Lambda}}=(P_{\gamma} )^{\hat{\Lambda}} $ on $(\Omega_F, \mathcal{G})$ and $P_{\gamma}= P_{({\gamma^{\Lambda}})^{\Lambda^{-1}}}=(P_{{\gamma} ^{\Lambda}})^{\hat{\Lambda}^{-1}} $ on $(\Omega_E, \mathcal{F})$, where $\hat \Lambda$ is defined in (\ref{hat lambda}). Hence $P_{\gamma^{\Lambda}}$ is the distribution of $\hat{\Lambda}(W_{\gamma})$. Furthermore, if
\begin{align}\label{define J}
J(y)=[h\circ \tau_1 ({\hat{\Lambda}}^{-1}(y))][h\circ \tau_k ({\hat{\Lambda}}^{-1}(y))], 
\end{align}
where  $h \in C_E^{*}[0,1]$ and $y \in \Omega_F$, then
\begin{align} \label{change variable J}
\int_{\Omega_F} J(y) dP_{{\gamma}^{\Lambda}}(y)= 
\int_{\Omega_F} J(y) dP_{\gamma}^{\hat{\Lambda}}(y) =\int_{\Omega_E} J({\hat{\Lambda}}(x))dP_{\gamma}(x)
\end{align}
\begin{align*}
= \int_{\Omega_E}[ h \circ \tau_{1}(x)][ h \circ \tau_{k}(x)]dP_{\gamma}(x).
\end{align*}
Moreover, for $y \in \Omega_F$ and $h \in C_E^{*}[0,1]$, $J(\cdot)$ is such that
\begin{align}\label{define J two}
J(y)= [ h \circ {\hat{\Lambda}_{1}}^{-1}({\hat{\tau}_1}(y)][ h \circ {\hat{\Lambda}_{1}}^{-1}({\hat{\tau}_k}(y)]~{\rm{with}}~ h \circ {\hat{\Lambda}_{1}^{-1}} \in C_F^{*}[0,1]. 
\end{align}
Hence, for $\hat{h}= h \circ {\hat{\Lambda}_{1}^{-1}}$ we have
\begin{align}\label{change J two}
 \int_{\Omega_F}[ \hat{h} \circ \hat{ \tau}_{1}(y)][ \hat{h} \circ \hat{\tau}_{k}(y)]dP_{\gamma^{\Lambda}}(y)= \int_{\Omega_E}[ h \circ \tau_{1}(x)][ h \circ \tau_{k}(x)]dP_{\gamma}(x).
\end{align}
 \end{subseclem}

\begin{proof} By definition of the Brownian motion induced by $\gamma^{\Lambda}$, for $0=t_0<t_1\cdots< t_n$ the random vectors
$$
y(t_j)-  y(t_{j-1}), j=1,\cdots ,n, ~{\rm{are~independent}}, 
$$
and $y(t_j)- y(t_{j-1})$ has distribution $\gamma_{t_j - t_{j-1}}^{\Lambda}$ on $F$, where $\gamma_{s}^{\Lambda}(A)= \gamma^{\Lambda}(A/\sqrt{s})$ for Borel subsets $A$ of $F$ when $s>0$ and $\gamma_0=\delta_0.$ In particular, the stochastic process  $W_{{\gamma}^{\Lambda}}=\{W_{{\gamma}^{\Lambda}}(t): t\ge 0\}$ defined on $(\Omega_F, \mathcal{G},P_{{\gamma}^{\Lambda}})$ by $W_{{\gamma}^{\Lambda}}(t,y)=y(t)$ has stationary independent mean zero Gaussian increments  and an increment of length $s>0$ has distribution given by ${{\gamma}^{\Lambda}} (\cdot/\sqrt{s})$.

We also have by definition that the stochastic process ${\hat{\Lambda}} (W_{\gamma})=: \{{\hat{\Lambda}} (W_{\gamma})(t): t \ge 0\}$ has law on $(\Omega_F,\mathcal{G})$ given by $(P_{\gamma})^{\hat{\Lambda}}$, where
\begin{align}\label{hat lambda two}
{\hat{\Lambda}} (W_{\gamma})(t)=\Lambda(W_{\gamma}(t)) , t \ge 0. 
\end{align}
Hence the increments of ${\hat{\Lambda}} (W_{\gamma})$ based on $0=t_0<t_1\cdots< t_n$ are 
\begin{align*}
{\hat{\Lambda}} (W_{\gamma})(t_j)-  {\hat{\Lambda}} (W_{\gamma})(t_{j-1})= \Lambda(W_{\gamma}(t_j) - W_{\gamma}(t_{j-1})),~ j=1,\cdots ,n.
\end{align*}
In addition, they are independent and an increment over an interval of length $s>0$ is such that 
$$
P_{\gamma}(x\in \Omega_E: \Lambda(x(t+s) -x(t)) \in A)= P_{\gamma}(x(1) \in {\Lambda}^{-1}(A/\sqrt{s}))= \gamma^{\Lambda}(A/\sqrt{s}), 
$$
where $A$ is any Borel subset of $F$. Thus $(P_{\gamma} )^{\hat{\Lambda}} = P_{\gamma^{\Lambda}}$ on $(\Omega_F, \mathcal{G})$. Furthermore, we therefore also have
(\ref{change variable J}) by the standard change of variables formula. Finally, given (\ref{define J}), (\ref{define J two}) is immediate from the definitions of the mappings, and 
(\ref{change J two}) follows from (\ref{change variable J}).
\end{proof}

Since $E$ is separable, the Banach-Mazur Theorem allows us to take the linearly isometric Banach space $F$ to be a subspace of the sup-norm Banach space $C[0,1]$ with norm $q_F(y)= \sup_{0\le t \le1}|y(t)|$. Furthermore, since $\gamma$ is a non-degenerate centered Gaussian measure on the Borel sets of $E$, its support is a non-degenerate closed linear subspace of $E$, and
${\gamma}^{\Lambda}$ is a Gaussian measure on the Borel subsets of $C[0,1]$ whose support is a closed linear subspace $F=\Lambda(E)$. Hence applying (\ref{define J two}) 
and (\ref{change J two}) to the conclusion of the following lemma, we will have a decay of the covariances suitable for every linear functional on $C_E[0,1]$. That is, the following lemma establishes the results for all linear functionals on $C_F[0,1]$, and since ${\hat{\Lambda}_{1}^{-1}}$ is an linear isometry from $ C_F[0,1]$ onto $C_E[0,1]$  the norms
\begin{align}\label{equal norms}
||h||^{*}=: ||h||_{C_{E}^{*}[0,1]}~{\rm{and}}~|| h \circ {\hat{\Lambda}_{1}^{-1}}||^{*}=: || h \circ {\hat{\Lambda}_{1}^{-1}}||_{C_{F}^{*}[0,1]} 
\end{align}
are equal (see \ref{hat lambda b}). Lemma \ref{decay rate lemma} below summarizes this, and hence the desired rate of decay for the covariances also holds for all linear functionals on $C_E[0,1]$.

\begin{subseclem}\label{change space lemma}  Let $\hat{\gamma}$ be a non-degenerate mean zero Gaussian measure on the sup-norm Banach space $C[0,1]$ whose support is a closed subspace $F$ of $C[0,1]$, and $Y_{\hat{\gamma}}=:\{Y_{\hat{\gamma}}(t): t \ge 0\}$ be the $F$-valued sample continuous Ornstein-Uhlenbeck process generated by $\hat{\gamma}$ on $F$ as in (\ref{gamma hat lambda OU}). Also assume the processes $\{\hat{X}_k(t): 0 \le t \le 1\}$ are as in (\ref{x k and hat x k}). Then, for each $\hat{h} \in C_F^{*}[0,1]$ and $k \ge  2$
\begin{align}\label{formula decay one}
 |E_{P_{\hat{\gamma}}}\big(\hat{h}(\hat{X}_1(\cdot))\hat{h}(\hat{X}_k(\cdot)\big)| \le e^{-\frac{k-2}{2}} \sigma^2({\hat{\gamma}}) ||{\hat h}||_{C_F^{*}[0,1]}^{2},
\end{align}
and 
\begin{align}\label{formula decay two}
E_{P_{\hat{\gamma}}} \big( \hat{h}^{2}(\hat{X}_1(\cdot))\big) \le ||{\hat h}||_{C_F^{*}[0,1]}^{2} \sigma^2({\hat{\gamma}}), 
\end{align}
where $\sigma^{2}(\hat{\gamma})$ is defined as in Lemma \ref{link gamma sigma}.
\end{subseclem}

\begin{proof} Identifying $C_{C[0,1]}[0,\infty)$ with $C([0,\infty)\times[0,1])$ we then have
$$
\Omega_F= \{y(\cdot,\cdot): y(\cdot,s) \in C_F[0,\infty), y(t,\cdot) \in F  \subseteq C[0,1], 0\le t <\infty \}
$$
$$
\subseteq \{y(\cdot,\cdot): y(\cdot,\cdot) \in C([0,\infty) \times[0,1])\},
$$
and by the Hahn-Banach theorem we let $\hat{\hat h}$ be a norm preserved extension of $\hat h$ from $C_{F}[0,1]$ to all of $C(I \times I)$, where $I=[0,1]$. Then, by the Riesz Representation theorem there is a signed measure $\lambda_{\hat{\hat  h}}$ on $I \times I$, where $\lambda_{\hat{\hat h}}$ has finite total variation $| \lambda_{\hat{\hat h}}|= ||{\hat{\hat h}}||_{C^{*}(I \times I)}=||{\hat{\hat h}}||_{C_F^{*}[0,1]} = ||\hat {h}||_{C_F^{*}[0,1]} $, and
\begin{align}\label{riesz formula}
{\hat{\hat  h}}(y(\cdot,\cdot))= \int_{I\times I} y(t,s) d\lambda_{\hat {\hat h}}(t,s), y(\cdot,\cdot) \in C(I \times I). 
\end{align}
Since we are assuming $Y_{\hat{\gamma}}$ is defined on $(\Omega_F,\mathcal{G},P_{\hat{\gamma}})$ as in (\ref{gamma hat lambda OU}), we have for $y \in \Omega_F, t \in [0,\infty), s\in [0,1],$ with  $P_{\hat{\gamma}}$-probability one that
\begin{align}\label{sample y hat gamma}
Y_{\hat{\gamma}}(t,y)(s) = e^{-\frac{t}{2}}y(e^t,s).
\end{align}
Using the notation defined in  (\ref{tau k and hat tau k}) and (\ref{x k and hat x k}), and identifying $y(t)(s)$ with $y(t,s)$, we have for  $k \ge 1, y \in \Omega_{F}$ that for $0\le t, s \le 1,$ with $P_{\hat{\gamma}}$-probability one 
\begin{align}\label{hat tau k = hat x k}
 \hat{\tau}_{k}(y)(t)(s) =  \hat{\tau}_k(y)(t,s))  =\hat{X}_k(t,y)(s)=Y_{{\hat{\gamma}}}(k-1+t,y)(s).  
\end{align}
Therefore,
\begin{align*}
E_{P_{\hat{\gamma}}}[\hat{h}(\hat{X}_1(\cdot))\hat{h}(\hat{X}_k(\cdot)]=E_{P_{\hat{\gamma}}}[{\hat{\hat h}}(\hat{X}_1(\cdot)) {\hat{\hat h}} (\hat{X}_k(\cdot)]=E_{P_{\hat{\gamma}}}[{\hat{\hat h}}(\hat{\tau_1}(y)){\hat{\hat h}}(\hat{\tau}_k(y))],
\end{align*}
and hence by (\ref{sample y hat gamma}) and (\ref{hat tau k = hat x k}) we have
\begin{align}\label{covariance and riesz}
E_{P_{\hat{\gamma}}}[\hat{h}(\hat{X}_1(\cdot)) \hat{h}(\hat{X}_k(\cdot)]=E_{P_{\hat{\gamma}}}[\int_{I\times I} \hat{\tau}_1(y)(t,s)d\lambda_{\hat{\hat h}}(t,s) \int_{I \times I} \hat{\tau}_k(y)(v,u)d\lambda_{\hat{\hat h}}(v,u)].
\end{align}

Moreover, since $\gamma$-Brownian motion has independent mean zero Gaussian increments,  
for $k \ge 2$ we have
$$
y(e^{k-1+v},\cdot)- y(e^{t},\cdot) ~{\rm {independent~of}}~ y(e^t,\cdot),
$$
and hence $y(e^{k-1+v},u)- y(e^{t},u)$ and $ y(e^t,s)$ are independent mean zero random variables. Therefore, the integrability of the Gaussian random variables and vectors involved implies
\begin{align}\label{covariance riesz inequality one}
 e^{\frac{t+k-1+v}{2}} |E_{P_{\hat{\gamma}}}[\hat{\tau}_1(y)(t,s)\hat{\tau}_k(y)(v,u)]|
\end{align}
$$
= |E_{P_{\hat{\gamma}}} \Big[y(e^t,s)[y(e^{k-1+v}, u) - y( e^t,u)] + y(e^t,s)y(e^t,u)\Big]|= |E_{P_{\hat{\gamma}}}\Big[ y(e^t,s)y(e^t,u)\Big]| 
$$
$$
\le\Big[E_{P_{\hat{\gamma}}}[ y^2(e^t,s)]\Big]^{\frac{1}{2}}  \Big[E_{P_{\hat{\gamma}}}[ y^2(e^t,u)]\Big]^{\frac{1}{2}} =e^{t} \Big[E_{P_{\hat{\gamma}}}[ y^2(1,s)]\Big]^{\frac{1}{2}}  \Big[E_{P_{\hat{\gamma}}}[ y^2(1,u)]\Big]^{\frac{1}{2}} 
$$
$$
= e^{t} \Big[\int_{F}y^{2}(1,s) d\hat{\gamma}(y) \Big]^{\frac{1}{2}}  \Big[\int_{F}[ y^2(1,u)] d{\hat{\gamma}}(y) \Big]^{\frac{1}{2}}\le  e^{t}\sigma^2({\hat{\gamma}}),
$$
where the last inequality holds by the definition of $ \sigma^2({\hat{\gamma}})$ as in Lemma \ref{link gamma sigma} and that the evaluation maps at $s,u$ are linear functionals of norm one. Therefore, for $t,v \in I$
\begin{align}\label{covariance riesz inequality two}
|E_{P_{\hat{\gamma}}}[\hat{\tau}_1(y)(t,s)\hat{\tau}_k(y)(v,u)]| \le e^{-\frac{k-1+v-t}{2}} \sigma^2({\hat{\gamma}}) \le e^{-\frac{k-2}{2}} \sigma^2({\hat{\gamma}}).
\end{align}

Combining  (\ref{covariance and riesz}) and (\ref{covariance riesz inequality two}), Fubini's theorem implies 
\begin{align}\label{covariance riesz inequality three}
|E_{P_{\hat{\gamma}}}[\hat{h}(\hat{X}_1(\cdot)) \hat{h}(\hat{X}_k(\cdot)]| \le  e^{-\frac{k-2}{2}} \sigma^2({\hat{\gamma}}) | \lambda_{\hat{\hat h}}|^2= e^{-\frac{k-2}{2}} \sigma^2({\hat{\gamma}}) ||{\hat h}||_{C_F^{*}[0,1]}^{2},  
\end{align}
and (\ref{formula decay one}) holds. The proof of (\ref{formula decay two}) is similar. That is, from (\ref{covariance and riesz}) we have
$$
E_{P_{\hat{\gamma}}}[\hat{h}^{2}(\hat{X}_1(\cdot))]= E_{P_{\hat{\gamma}}} \Big [ [\int_{I\times I} \hat{\tau}_1(y)(t,s)d\lambda_{\hat{\hat h}}(t,s)]^{2}\Big]
$$
$$
\le E_{P_{\hat{\gamma}}} \Big [ [\int_{I\times I} | \hat{\tau}_1(y)(t,s)|d|\lambda_{\hat{\hat h}}|(t,s)]^{2}\Big] \le E_{P_{\hat{\gamma}}} \Big [ [\int_{I\times I} e^{-t}y^{2}(e^{t} ,s) d|\lambda_{\hat{\hat h}}|(t,s)]\Big] |\lambda_{\hat{\hat h}}|(I \times I)
$$
$$
=\int_{I\times I} e^{-t} E_{P_{\hat{\gamma}}}[y^{2}(e^{t} ,s)] d|\lambda_{\hat{\hat h}}|(t,s) |\lambda_{\hat{\hat h}}|(I \times I)=\int_{I\times I}  E_{P_{\hat{\gamma}}}[y^{2}(1 ,s)] d|\lambda_{\hat{\hat h}}|(t,s) |\lambda_{\hat{\hat h}}|(I \times I)
$$
$$
\le [ |\lambda_{\hat{\hat h}}|(I \times I)]^{2}   \sigma^2({\hat{\gamma}})  = ||{\hat h}||_{C_F^{*}[0,1]}^{2}  \sigma^2({\hat{\gamma}}),
$$
which establishes (\ref{formula decay two}).
\end{proof}

\begin{subseclem}\label{decay rate lemma} Let $Y_{\gamma}=: \{Y_{\gamma}(t): t \ge 0\}$ be an $E$-valued sample continuous Ornstein-Uhlenbeck process generated by the non-degenerate mean zero Gaussian measure $\gamma$ on $E$ as in (\ref{gamma OU}), and assume the processes $\{X_k(t): 0 \le t \le 1\}$ are as in (\ref{x k and hat x k}). Then, for each $h \in C_E^{*}[0,1]$  and $k \ge 2$
\begin{align}\label{decay rate formula three}
 |E_{P_{\gamma}}\big(h(X_1(\cdot))h(X_k(\cdot)\big)| \le  e^{-\frac{k-2}{2}} \sigma^2({\gamma}) ||h||_{C_E^{*}[0,1]}^{2}, 
\end{align}
and 
\begin{align}\label{decay rate formula four}
 E_{P_{\gamma}} \big( h^{2}(X_1(\cdot))\big) \le  ||h||_{C_E^{*}[0,1]}^{2}  \sigma^2({\gamma}).
\end{align}
\end{subseclem}

\begin{proof} Since $E$ is a separable Banach space, the Banach-Mazur theorem shows there is a linear isometry $\Lambda$ mapping $E$ onto $F$, where $F$ is a closed subspace of the sup-norm Banach space $C[0,1]$, and the $q$-norm on $E$ maps to the sup-norm on $F$. Using  (\ref{change variable J}) and (\ref{define J two}) of Lemma \ref{change variables} with $\hat{\gamma} = \gamma^{\Lambda}$ and $P_{\hat{\gamma}} = P_{\gamma^{\Lambda}}$, we then have for $k\ge 2$ that 
\begin{align}\label{covariance equality two}
E_{P_{\gamma}}\big(h(X_1(\cdot))h(X_k(\cdot)\big)=E_{P_{\gamma^{\Lambda}} }\Big[ [h\circ \tau_1 ({\hat{\Lambda}}^{-1}(y))][h\circ \tau_k ({\hat{\Lambda}}^{-1}(y))] \Big] 
\end{align}
\begin{align*}
= E_{P_{\gamma^{\Lambda}} }\Big[ [h\circ {\hat{\Lambda}}^{-1}_{1} ({\hat{\tau}}_1 (y))][  h\circ {\hat{\Lambda}}^{-1}_{1} ({\hat{\tau}}_k (y))] \Big] 
=E_{P_{\hat{\gamma}} }\Big[ [\hat{h}({\hat{\tau}}_1 (y))][ \hat{h}({\hat{\tau}}_k (y))] \Big], 
\end{align*}
where $\hat{h}= h\circ {\hat{\Lambda}}^{-1}_{1}$. Combining (\ref{formula decay one}) of Lemma \ref{change space lemma} and (\ref{covariance equality two}) we have
\begin{align}\label{decay rate five}
E_{P_{\gamma}}\big(h(X_1(\cdot))h(X_k(\cdot)\big)  \le e^{-\frac{k-2}{2}} \sigma^2({\hat{\gamma}}) ||{\hat h}||_{C_F^{*}[0,1]}^{2}= e^{-\frac{k-2}{2}} \sigma^2(\gamma) ||h||_{C_E^{*}[0,1]}^{2},
\end{align}
where the equality in (\ref{decay rate five}) holds since(\ref{equal norms}) implies
$$
||\hat{h}||_{C_F^{*}[0,1]} =   || h \circ {\hat{\Lambda}_{1}^{-1}}||_{C_{F}^{*}[0,1]} = ||h||_{C_E^{*}[0,1]}, 
$$
and $\sigma^{2}(\gamma)= \sigma^{2}({\hat{\gamma}})$. Hence (\ref{decay rate formula three}) holds for $k \ge 2$, and the proof of (\ref{decay rate formula four}) is entirely similar. 
\end{proof} 

\begin{subsecrem}\label{Gamma=gamma} If $V= \{h \in C_{E}^{*}[0,1]: ||h||_{C_{E}^{*}[0,1] }\le 1\},$ then (\ref{decay rate formula four}) implies
\begin{align}\label{decay rate formula six}
\sup_{h \in V} E_{P_{\gamma}} \big( h^{2}(X_1(\cdot))\big) = \sup_{h \in V } \int_{C_E[0,1]} h^{2}(\tau_1(x))dP_{\gamma}(x) \le   \sigma^2({\gamma})= \Gamma^{2}_{\gamma}, 
\end{align}
where $\Gamma_{\gamma}$ is defined as in Lemma \ref{link gamma sigma}  and the equality follows from (\ref{gamma=sigma}). Furthermore, the inequality in (\ref{decay rate formula six}) is actually an equality, which can be seen as follows. That is, let $z \in C_E[0,1]$ and for $f \in E^{*}$ define $h_0 \in C_E^{*}[0,1]$ to be
$h_0(z)= f(z(0)).$ Hence, if $U= \{ z \in C_E[0,1]: \sup_{0\le t \le 1} q_E(z(t)) \le 1\},$ then $z \in U$ implies $q_E(z(0)) \le 1,$ and 
$||h_0||_{C^{*}_E[0,1]}=\sup_{z \in U} | h_0(z)| = \sup_{z \in U} | f(z(0)) |
\le 1$ when $||f||_{E^{*}}=1$. Letting $J=: \sup_{ ||h||_{C_E^{*}}[0,1] \le 1} E_{P_{\gamma}} \big( h^{2}(\tau_1(x))\big),$
we see
$$
J \ge \int_{C_E[0,1]} h^{2}_0(\tau_1(x))dP_{\gamma}(x) = \int_{C_E[0,1]} f^{2}(\tau_1(x)(0))dP_{\gamma}(x)= \int_{E} f^{2}_0(w)d{\gamma}(w).
$$

Therefore, for  $f=f_0$ as in Lemma \ref{link gamma sigma} with the Gaussian measure being $\gamma$, we have $||f_0||_{E^{*}}=1$ and $J \ge  \int_{E} f^{2}_0(x)d{\gamma}(x)= \sigma^{2}(\gamma)= \Gamma_{\gamma}^{2}.$
Thus equality holds in (\ref{decay rate formula six}) as claimed.
\end{subsecrem}

\bibliographystyle{amsalpha} 
\bibliography{mybib11-3-17,othbib2-20-18}

\end{document}